\documentclass[12pt]{amsart}
\usepackage{amscd,amssymb,amsmath,verbatim,amsthm}

\newtheorem{thm}{Theorem}[section]
\newtheorem{cor}[thm]{Corollary}
\newtheorem{lem}[thm]{Lemma}
\newtheorem{prop}[thm]{Proposition}
\newtheorem{claim}[thm]{Claim}

\theoremstyle{remark}
\newtheorem{remark}[thm]{Remark}
\numberwithin{equation}{section}

\newcommand{\R}{\mathbb R}
\newcommand{\Z}{\mathbb Z}

\newcommand{\ric}{\mathrm{Ric}}

\newcommand{\vol}{\mathrm{Vol}}
\newcommand{\const}{\mathrm{const}}

\newcommand{\diam}{\mathrm{diam}}
\newcommand{\Hess}{\mathrm{Hess}}
\newcommand{\rem}{\mathrm{Rm}}
\newcommand{\injrad}{\mathrm{inj}}
\newcommand{\Diff}{\mathrm{Diff}}
\newcommand{\Nil}{\mathrm{Nil}}
\newcommand{\Sol}{\mathrm{Sol}}

\newcommand{\SL}{\mathrm{SL}}
\newcommand{\SO}{\mathrm{SO}}
\newcommand{\Tr}{\mathrm{Tr}}
\newcommand{\Aut}{\mathrm{Aut}}

\def\XXint#1#2#3{{\setbox0=\hbox{$#1{#2#3}{\int}$}
     \vcenter{\hbox{$#2#3$}}\kern-.5\wd0}}

\begin{document}
\title{Ricci flow on three-dimensional manifolds with symmetry}
\author{John Lott and Natasa Sesum}
\address{Department of Mathematics\\
University of California - Berkeley\\
Berkeley, CA  94720-3840\\
USA} \email{lott@math.berkeley.edu}
\address{Department of Mathematics\\
Rutgers University\\
Piscataway, NJ 08854-8019\\
USA} \email{natasas@math.rutgers.edu}

\thanks{The research of the first author was partially supported by
NSF grant DMS-0900968.  The research of the second author was
partially supported by NSF grant DMS-1110145}
\date{October 4, 2011}

\begin{abstract}
We describe the Ricci flow on two classes of
compact three-dimensional
manifolds :\\
1. Warped products with a circle fiber
over a two-dimensional base.\\
2. Manifolds with a free local isometric
$T^2$-action.
\end{abstract}

\maketitle

\section{Introduction} \label{sec1}

In understanding three-dimensional Ricci flow
solutions, a special role is played by three-manifolds with symmetry.
It is plausible that one can get more precise results about such
Ricci flows than in the general case.  Consequently, such
Ricci flows with symmetry have been studied for quite a while.
We begin by describing some of the earlier results.

Locally homogeneous three-dimensional Ricci flow solutions were examined by
Isenberg-Jackson \cite{Isenberg-Jackson} and Knopf-McLeod \cite{Knopf-Mcleod}.
The flow equations reduced to a system of three coupled ODEs.
The solutions are now fairly well understood;
see, for example, \cite[Section 3]{L0}.

Certain three-dimensional
Ricci flow solutions with a two-dimensional isometry group
were analyzed by Carfora-Isenberg-Jackson 
\cite{CIJ} and Hamilton \cite[Section 11]{Ha}.
They considered Riemannian
three-manifolds that admit a free isometric $T^2$-action. The base is a circle
and the total space is (necessarily) diffeomorphic to a $3$-torus.
Under some additional assumptions
(${\mathcal F}$-metric \cite{CIJ} or square torus metric 
\cite[Section 11]{Ha})
it was shown that the Ricci 
flow exists for all time and converges to a flat metric.

Hamilton and Isenberg considered a twisted version of such
torus bundles \cite{Hamilton-Isenberg}. That is, there was a local
(with respect to the base) free isometric $T^2$-action. The
$T^2$-bundle over the circle was globally twisted by a hyperbolic element of
$\SL(2, \Z)$; see Subsection \ref{subsec3.1} for a more precise description.
Under the additional assumption of a ``solv-Gowdy'' metric, 
it was shown that the
Ricci flow approaches that of a locally homogeneous $\Sol$-metric;
see also \cite{Knopf}.

Passing to one-dimensional isometry groups, a natural class
of geometries comes from warped product metrics
with circle fibers and a closed surface base $M$. One starts
with a product metric on $N = M \times S^1$ and then allows the
circle length over $m \in M$ to become $m$-dependent. That is,
we consider Riemannian metrics $h$ on $N$ of the form
\begin{equation} \label{1.1}
h = g + e^{2u} d\theta^2,
\end{equation}
where $g$ is a Riemannian metric on $M$, $u \in C^\infty(M)$ and
$\theta$ is the standard coordinate on $S^1$.
It is not hard to see that the Ricci flow preserves the
warped product structure. The Ricci flow equation on $N$ becomes
two coupled evolution equations on $M$ for $g(t)$ and $u(t)$.
The evolution equation for $g(t)$ is like the Ricci flow equation 
on the surface $M$, but there is an extra term involving $u(t)$.
Because of this extra term, the techniques used to
analyze Ricci flow on surfaces
\cite[Chapter 5]{Chow-Knopf} break down.

When $M$ is a two-sphere, Xiaodong Cao showed that the product of $\R$
and a cigar soliton cannot arise as a finite-time dilation limit
of a warped product Ricci flow on $S^2 \times S^1$.
\cite{Cao}. His argument used an isoperimetric inequality.
(Cao's work preceded Perelman's proof that the product of $\R$
and a cigar soliton can never arise as a finite-time dilation limit
for the Ricci flow on a compact three-manifold \cite{Perelman}.)

In the present paper we use new techniques to
give general results about the Ricci flow
on three-manifolds with symmetries.  One of our tools is the
result of \cite{L} giving the long-time behavior of a three-dimensional
Ricci flow solution $(N, h(\cdot))$
satisfying $\max_{p \in N} |\rem^N|(p,t) =
O(t^{-1})$ and $\diam(N, h(t)) = O(\sqrt{t})$. Thus one of our main
goals is to show that these bounds are satisfied in the relevant cases.
In particular, we show that $\max_{p \in N} |\rem^N|(p,t) =
O(t^{-1})$ for all of the immortal Ricci flows under consideration.
(A Ricci flow solution is immortal if it exists for $t \in [0, \infty)$.)
It is an open question whether an immortal Ricci flow solution on a
compact three-manifold always satisfies this curvature bound.

To describe the results of the paper, we
start with warped products.
In the special case of a product metric, the
Ricci flow solution is the isometric product of $S^1$ with
a Ricci flow solution on the base $M$. Such two-dimensional
Ricci flows are well understood \cite[Chapter 5]{Chow-Knopf}.
If $h(0)$ is a warped product metric on $N  = M \times S^1$
then a natural conjecture is that the Ricci flow
asymptotically approaches a product flow.

\begin{thm} \label{thm1.1}
Let $h(\cdot)$ be a Ricci flow solution on a closed connected 
orientable three-dimensional manifold $N$. Suppose that $h(0)$ is a
warped product metric as in (\ref{1.1}) with 
a two-dimensional orientable base $M$.
\begin{enumerate}
\item[(i)]
If $\chi(M) > 0$ then there is a finite singularity time $T<\infty$.
As $t \rightarrow T^-$, the lengths of the circle fibers remain
uniformly bounded above and below.  For any $p \in N$, the pointed
smooth limit
$\lim_{t \rightarrow T^-} \left( N, p, \frac{1}{T-t} h(t) \right)$
exists and is the isometric product of $\R$ with a sphere
$S^2$ of constant curvature $\frac12$.

\item[(ii)] If $\chi(M) \le 0$ then the Ricci flow exists for all
$t \in [0, \infty)$. Also, there is a constant $C < \infty$
so that for all $p \in N$ and $t \in [0, \infty)$, one has
$\left| \rem^N \right| (p,t) \le \frac{C}{t}$.

\item[(iii)]
If $\chi(M) = 0$ then $\lim_{t \rightarrow \infty} h(t)$
exists and is a flat metric on $T^3$. The convergence is
exponentially fast.

\item[(iv)] If $\chi(M) < 0$, put $\widehat{g}(t) = \frac{g(t)}{t}$.
For any $i_0 > 0$, 
define the $i_0$-thick part of $(M, \widehat{g}(t))$ by
\begin{equation} \label{1.2}
X_{i_0}(t) = 
\{m \in M \, : \, \injrad_{\widehat{g}(t)}(m) \ge i_0\}.
\end{equation}
Then
\begin{equation} \label{1.3}
\lim_{t \rightarrow \infty} \max_{x \in X_{i_0}(t)} |R_{\widehat{g}(t)}(x) + 1|
= 0
\end{equation}
and
\begin{equation} \label{1.4}
\lim_{t \rightarrow \infty} \max_{x \in X_{i_0}(t)} 
|\widehat{\nabla} u|_{\widehat{g}(t)}(x)
= 0.
\end{equation}
For all sufficiently small $i_0$, if $t$ is sufficiently large
then $X_{i_0}(t)$ is nonempty.
\end{enumerate}
\end{thm}

\begin{remark} \label{rem1.2}
The proof of Theorem \ref{thm1.1}.(i) is essentially contained in List's paper
\cite{Li} on a modified Ricci flow.  The only thing missing from 
\cite{Li} is the
observation that his flow differs from the warped product flow by
a Lie derivative.
\end{remark}

\begin{remark} \label{rem1.3}
The proof of the curvature bound in
Theorem \ref{thm1.1}.(ii) is by contradiction, using a blowup
argument, a sharp volume estimate and the Gauss-Bonnet theorem.
\end{remark}

\begin{remark} \label{rem1.4}
The proof of Theorem \ref{thm1.1}.(iii) is somewhat indirect.
We first show that $\vol(M, g(t)) = O(\sqrt{t})$ and that the
length of the shortest noncontractible curve on $M$ is nondecreasing
in $t$.
A geometric argument then shows that $\diam(M, g(t)) = O(\sqrt{t})$.
From \cite{L}, we deduce that 
\begin{equation}
\left( \max_{p \in N} |\rem^N|(p,t) \right) 
\cdot \diam^2(N, h(t)) \: = \: o(t).
\end{equation}
Rescaling at a given time $t$ to diameter one, 
if $t$ is sufficiently large then 
we can assume that $\max_{p \in N} |\rem^N|(p,t)$ is arbitrarily small.  
After passing to a finite
cover $\widehat{N}$, we can assume that there is a universal 
lower bound on the injectivity radius of the pullback metric
$\widehat{h}(t)$. By the linear stability
of flat metrics \cite{GIK}, 
$\lim_{t \rightarrow \infty} \widehat{h}(t)$ exists
and is a flat metric. Hence
$\lim_{t \rightarrow \infty} h(t)$ exists
and is a flat metric.
\end{remark}

\begin{remark} \label{rem1.5}
Theorem \ref{thm1.1}.(iv) says that in the case $\chi(M) < 0$
and as $t \rightarrow \infty$,
over a large part of $M$ the flow $(N, h(t))$ approaches
a product flow of
$S^1$ times a finite-volume surface of constant sectional
curvature $- \: \frac{1}{2t}$.
Our result here is possibly nonoptimal;
see Remark \ref{rem2.20}.
\end{remark}

Our other main result is about Ricci flow solutions with a 
local $U(1) \times U(1)$ symmetry.

\begin{thm} \label{thm1.6}
Let $N$ be an orientable three-manifold that fibers over $S^1$ with
$T^2$-fibers.
Choosing an
orientation for $S^1$, let $H \in \SL(2, \Z) = \pi_0(\Diff^+(T^2))$ be the
holonomy of the torus bundle.
We can consider $N$ to be the total space of a 
twisted principal $U(1) \times U(1)$
bundle, where the twisting is determined by $H$.

Let $h(\cdot)$ be a Ricci flow solution on $N$.
Suppose that $h(0)$ is invariant under the local $U(1) \times U(1)$ actions.
Then the Ricci flow exists for all
$t \in [0, \infty)$. There is a constant $C < \infty$
so that for all $p \in N$ and $t \in [0, \infty)$, one has
$\left| \rem^N \right| (p,t) \le \frac{C}{t}$.

\begin{enumerate}
\item[(i)]
If $H$ is elliptic, i.e. has finite order, then
$\lim_{t \rightarrow \infty} h(t)$
exists and is a flat metric on $N$.
The convergence is exponentially fast.

\item[(ii)]
Suppose that $H$ is hyperbolic, i.e. has two distinct real eigenvalues.
We write $h(t)$ in the form
\begin{equation}
h(t) \: = \: g_{yy}(y,t) \: dy^2 \: + \: (dx)^T G(y,t) dx,
\end{equation}
where $\{x^1, x^2\}$ are local coordinates on $T^2$ and
$y \in [0, 1)$ is a local coordinate on $S^1$.
Then up to an overall change of parametrizations for $S^1$ and $T^2$, 
we have
\begin{align} \label{conv}
\lim_{t \rightarrow \infty} \frac{g_{yy}(y,t)}{t} & \: = \: \frac12 \: \Tr(X^2), \\
\lim_{t \rightarrow \infty} G(y,t) & \: = \: e^{yX}, \notag
\end{align}
where $X$ is the real symmetric matrix such that $e^X = H^T H$.
The convergence in (\ref{conv}) is power-decay fast in $t$.
\end{enumerate}
\end{thm}

\begin{remark}
Theorem \ref{thm1.6}.(ii) says that $(N, h(\cdot))$ approaches a
locally homogeneous Ricci flow solution of $\Sol$-type.
The proof of Theorem \ref{thm1.6}.(ii) uses the monotonicity
of the modified $W_+$-functional from
\cite[Section 4.2.3]{L}, along with the local stability result of
\cite[Theorem 3]{Knopf2}.
\end{remark}

\begin{remark} \label{rem1.7}
Theorem \ref{thm1.6}.(i) includes the cases considered in
\cite{CIJ} and \cite[Section 11]{Ha}. Theorem
\ref{thm1.6}.(ii) includes the case considered in
\cite{Hamilton-Isenberg}. 
\end{remark}

The structure of the paper is as follows.  In Section \ref{sec2}
we prove Theorem \ref{thm1.1}. In Section \ref{sec3} we prove 
Theorem \ref{thm1.6}.  More detailed descriptions are at the
beginnings of the sections.

We are grateful to Jim Isenberg for introducing us to these
problems and for sharing his knowledge with us. 
We learned the estimates (\ref{2.19}), (\ref{2.24}) and 
(\ref{2.26}) from Jim.

In what follows we will use the Einstein summation convention freely.

\section{Warped products} \label{sec2}

In this section we prove Theorem \ref{thm1.1}. In Subsection \ref{subsec2.1} we
write the warped product Ricci flow equations and the corresponding
evolution equations for geometric quantities.  In Subsection
\ref{subsec2.2} we give {\em a priori} bounds using the maximum principle
and integral estimates. In Subsection \ref{subsec2.3} we deal with the
case $\chi(M) > 0$. In Subsection \ref{subsec2.4}
we prove that the Ricci flow exists for
$t \in [0, \infty)$ when $\chi(M) \le 0$.
In Subsection \ref{subsec2.5} we deal with the
case $\chi(M) = 0$.  In Subsection \ref{subsec2.6} we deal with the
case $\chi(M) < 0$.

\subsection{Warped product Ricci flow} \label{subsec2.1}

Let $N = M \times S^1$ be endowed with a warped product metric
\begin{equation} \label{2.1}
h = g + e^{2u} d\theta^2.
\end{equation}
Here $g$ is a Riemannian metric on $M$, $u$ is a smooth function on $M$ and
$\theta \in [0, 2\pi)$.
We use $i,j,k,l$\ for the indices on $M$.
Nonzero components of the curvature tensor of $(N, h)$
are
\begin{align} \label{2.2}
R^N_{ijkl} & = R^M_{ijkl}, \\
R^N_{\theta i \theta j} & = - e^{2u} \left[ \nabla_i \nabla_j u +
(\nabla_i u)(\nabla_j u) \right]. \notag
\end{align}
Its square norm is
\begin{align} \label{2.3}
\left| \rem^N_h \right|^2  = & 
\left| \rem^M_g \right|^2 + \\
& 2 g^{ii'} g^{jj'} \left[ \nabla_i \nabla_j u +
(\nabla_i u)(\nabla_j u) \right] \left[ \nabla_{i'} \nabla_{j'} u +
(\nabla_{i'} u)(\nabla_{j'} u) \right]. \notag
\end{align}
We will often think of $\left| \rem^N_h \right|(\cdot, t)$ as a function
on $M$ rather than $N$, since it pulls back from $M$, and write
$\left| \rem^N_h \right|(m,t)$.
Note that a sectional curvature bound on $(N, h)$ implies the same
sectional curvature bound on $(M, g)$.

The nonzero components of the Ricci tensor are
\begin{align} \label{2.4}
\ric^N_{ij} & = \ric^M_{ij}  - \nabla_i \nabla_j u -
(\nabla_i u) (\nabla_j u),  \\
\ric^N_{\theta \theta} & = - e^{2u} \left( \triangle u + |\nabla u|^2
\right).  \notag
\end{align}
The scalar curvature is
\begin{equation} \label{2.5}
R^N = R^M - 2 \triangle u - 2|\nabla u|^2.
\end{equation}

The Ricci flow equation
\begin{equation} \label{2.6}
\frac{dh}{dt} = - \, 2 \,\ric_{h(t)},
\end{equation}
on $N$ preserves the warped product structure.
In terms of $g$ and $u$, it becomes
\begin{align} \label{2.7} 
\frac{\partial g_{ij}}{\partial t}  &  =- 2 R_{ij}+
2 \nabla_i \nabla_j u + 2 (\nabla_i u) (\nabla_j u), \\
\frac{\partial u}{\partial t}  &  = \Delta u + |\nabla u|^2, \notag
\end{align}
where $\Delta=\Delta_{g(t)}$.
Adding the Lie derivative with respect to $- \nabla u$ to the
right-hand side of (\ref{2.7}) gives the modified equations
\begin{align} \label{2.8} 
\frac{\partial g_{ij}}{\partial t}  &  =- 2 R_{ij}+
2 (\nabla_i u) (\nabla_j u), \\
\frac{\partial u}{\partial t}  &  = \Delta u. \notag
\end{align}
Hereafter we will mainly work with the system (\ref{2.8}), since geometric
statements about (\ref{2.8}) will imply the corresponding statements
about the Ricci flow (\ref{2.7}).

Given $s > 0$ and a solution $(g(\cdot), u(\cdot))$ of (\ref{2.8}), 
we obtain another solution $(g_s(\cdot), u_s(\cdot))$  of (\ref{2.8}) by
putting 
\begin{align} \label{2.9}
g_s(t) & = \frac{1}{s} g(st), \\
u_s(t) & = u(st). \notag
\end{align}
Note that the rescaling in (\ref{2.9}) differs from the three-dimensional
rescaling of the Ricci flow $(N, h(\cdot))$, which would give
\begin{align} \label{2.10}
G_s(t) & = \frac{1}{s} g(st), \\
U_s(t) & = u(st) - \, \frac12 \ln(s). \notag
\end{align}
The rescaling in (\ref{2.9}) can be interpreted in the following way.
There is a $\Z$-cover $\widehat{N} = M \times \R$ of $N$ with a pullback Ricci
flow $\widehat{h}(\cdot)$. On $\widehat{N}$, the rescaling in (\ref{2.9})
amounts to looking at the Ricci flow solution
$\frac{1}{s} \phi_s^* \widehat{h}(st)$, where $\phi_s(m,r) = 
\left( m, \sqrt{s} r \right)$. We refer to 
\cite[Section 4]{L0} for further discussion of this point.

Put 
\begin{equation} \label{2.11}
h_s = g_s + e^{2u_s} d\theta^2.
\end{equation}
One sees from (\ref{2.3}) and (\ref{2.9}) that
\begin{equation} \label{2.12}
\left| \rem_{h_s} \right|^2 = s^2 \left| \rem_h \right|^2. 
\end{equation}
This can also be seen from
the discussion of the previous paragraph on
$\widehat{N}$ along with the fact that under the three-dimensional
rescaling of (\ref{2.10}), the norm square of the curvature gets multiplied
by $s^2$.

Following \cite{Li}, we introduce the tensor
\begin{equation} \label{2.13}
S_{ij} = R_{ij} - (\nabla_i u) (\nabla_j u)
\end{equation}
and its trace
\begin{equation} \label{2.14}
S = R - |\nabla u|^2.
\end{equation}
Putting $\alpha_n = 1$ in \cite[Lemma 3.2]{Li} gives
\begin{equation} \label{2.15}
\frac{\partial |\nabla u|^2}{\partial t} = \triangle |\nabla u|^2 -
2 |\Hess (u)|^2 - 2 |\nabla u|^4
\end{equation}
and
\begin{equation} \label{2.16}
\frac{\partial S}{\partial t} = \triangle S + 2 |S_{ij}|^2 + 2
|\triangle u|^2.
\end{equation}

Hereafter we specialize to the case when $M$ is closed,
connected and orientable, with $\dim(M) = 2$. Then (\ref{2.8})
simplifies to
\begin{align} \label{2.17} 
\frac{\partial g_{ij}}{\partial t}  &  =- R  g_{ij}+
2 (\nabla_i u) (\nabla_j u), \\
\frac{\partial u}{\partial t}  &  = \Delta u. \notag
\end{align}
Note that if $u$ is nonconstant then the conformal class of $g(t)$
is $t$-dependent.

\subsection{{\it A priori} bounds} \label{subsec2.2}

\subsubsection{Bounds from the maximum principle} \label{subsubsec2.2.1}

We take the Ricci flow to start at time zero.
Applying the maximum principle to the second equation in
(\ref{2.8}) shows that  there exist constants $C_1, C_2 > 0$ so that
for all $m \in M$ and all $t$ for which the flow exists,
\begin{equation} \label{2.18}
C_1 \le u(m, t) \le C_2.
\end{equation}

Next, applying the maximum principle to (\ref{2.15}) implies
(c.f. \cite[Lemma 5.6]{Li}) that
\begin{equation} \label{2.19}
|\nabla u|^2(\cdot,t) \le \frac{c}{2ct + 1},
\end{equation}
where $c = \max_{p \in M} |\nabla u|^2(p,0)$. 
Similarly, applying the maximum principle to (\ref{2.16}) implies that
\begin{equation} \label{2.20}
S(\cdot,t) \ge - \frac{1}{t}.
\end{equation}

In particular,
\begin{equation} \label{2.21}
R(\cdot,t) \ge - \frac{1}{t}.
\end{equation}

\subsubsection{Volume estimates} \label{subsubsec2.2.2}

Let
$V(t)$ denote the volume of $(M^{2},g(t))$. 
As
\begin{equation} \label{2.22}
\frac{d}{dt} dV_{g(t)} = 
\frac{1}{2} g^{ij} \frac{\partial g_{ij}}{\partial t}
dV_{g(t)} =  \left( -R + |\nabla u|^2 \right) dV_{g(t)},
\end{equation}
we have
\begin{equation} \label{2.23}
\frac{dV}{dt}  
 =-4\pi\cdot\chi(M)+\int_{M} |\nabla u|^2 \, dV_{g(t)},
\end{equation}
where $\chi(M)$ is the Euler characteristic of $M$.
Thus $\frac{dV}{dt} \geq-4\pi\cdot\chi(M)$
and so
\begin{equation} \label{2.24}
V(t)\geq-4\pi\cdot\chi(M)\cdot t+V(0).
\end{equation}

Using (\ref{2.19}) in addition, we can control $V(t)$ from above.
Namely, from (\ref{2.19}),
\begin{equation} \label{2.25}
\frac{dV}{dt}
\leq-4\pi\cdot\chi(M)+\frac{c}{2ct+1}\cdot V
\end{equation}
where $c$ is as in (\ref{2.19}).
Then by ODE comparison,
\begin{equation} \label{2.26}
V(t)\leq-\frac{4\pi}{c}\cdot\chi(M) \cdot (2ct+1)+\sqrt{2ct+1}\cdot
\left( \frac{4\pi
\cdot\chi(M)}{c}+V(0) \right).
\end{equation}

If $\chi(M) = 0$ then estimates (\ref{2.24}) and (\ref{2.26}) imply
that
\begin{equation} \label{2.27}
V(0) \le V(t) \le C(1 + \sqrt{t})
\end{equation}
for an appropriate constant $C < \infty$.

When $\chi(M) < 0$,
the linear term on the right-hand side of (\ref{2.26}) is
$- 8\pi \chi(M)t$. We would like to improve this to
$- 4 \pi \chi(M)t$, to bring it in line with (\ref{2.24}). 
Put $E(t) = \int_M |\nabla u|^2 \, dV_{g(t)}$.  

\begin{lem} \label{lem2.1}
We have
\begin{equation} \label{2.28}
\frac{d E}{dt} \le -\frac{E^2(t)}{V(t)}.
\end{equation}
\end{lem}
\begin{proof}
First, as in (\ref{2.22}), the volume density changes by
\begin{equation} \label{2.29}
\frac{d}{dt} (dV_{g(t)}) = - S dV_{g(t)}.
\end{equation} 
Note that
\begin{equation} \label{2.30}
\int_M S \, dV_{g(t)} = \int_M (R - |\nabla u|^2) \, dV_{g(t)} =
4\pi \chi(M) - E(t).
\end{equation}
Then
\begin{align} \label{2.31}
-\frac{dE}{dt} =& \: \frac{d}{dt}\int_M S\, dV_{g(t)} = \int_M (\Delta S + 2|S_{ij}|^2 + 2(\Delta u)^2 - S^2)\, dV_{g(t)} \\
=& \: 2\int_M \left| S_{ij} - \frac{1}{2}S g_{ij} \right|^2\, 
dV_{g(t)} + 2\int_M (\Delta u)^2\, dV_{g(t)} \notag \\
=& \int_M \left( 2 \left| - \nabla_i u\nabla_j u +
\frac{1}{2}|\nabla u|^2 g_{ij} \right|^2 
+ 2(\Delta u)^2\right)\, dV_{g(t)} \notag \\
=& \int_M (|\nabla u|^4 + 2(\Delta u)^2)\, dV_{g(t)}  \notag \\
\ge& \int_M |\nabla u|^4\, dV_{g(t)} \notag \\
\ge& \: \frac{1}{V(t)}\left( \int_M |\nabla u|^2\, dV_{g(t)}\right)^2,
\notag
\end{align}
which proves the lemma.
\end{proof}

We now show that $E(t)$ decays logarithmically in time
when $\chi(M) < 0$.

\begin{cor} \label{cor2.2}
If $\chi(M) < 0$ then there exist constants $A, B > 0$ so that
\begin{equation} \label{2.32}
\int_M |\nabla u|^2\, dV_{g(t)} \le \frac{A}{1 + B\ln (t+1)}
\end{equation}
at all $t$ for which the flow exists.
\end{cor}

\begin{proof}
By Lemma \ref{lem2.1} and the volume estimate (\ref{2.26}) we have
\begin{equation} \label{2.33}
\frac{d E}{dt} \le -\frac{E^2(t)}{V(t)} \le -\frac{E^2(t)}{c_1 t + c_2}
\end{equation}
for appropriate constants $c_1, c_2 > 0$.
The corollary follows from ODE comparison.
\end{proof}

\begin{cor} \label{cor2.3}
If $\chi(M) < 0$ then there is a function $\alpha : [0, \infty) \rightarrow
[0, \infty)$, with $\lim_{t \rightarrow \infty} \alpha(t) = 0$, such that
\begin{equation} \label{2.34}
\left| \frac{V(t)}{-4\pi\chi(M) \cdot (t+1)} - 1 \right| \le \alpha(t)
\end{equation} 
at all $t$ for which the flow exists.
\end{cor}
\begin{proof}
This follows from (\ref{2.23}) and (\ref{2.32}).
\end{proof}

\begin{lem} \label{lem2.4}
The quantity $\frac{V(t)}{t}$ is 
nonincreasing along the flow (\ref{2.8}).
\end{lem}

\begin{proof}
From (\ref{2.20}) and (\ref{2.29}),
\begin{equation} \label{2.35}
\frac{d}{dt} \left( \frac{V(t)}{t} \right) = - \, 
\frac{1}{t}\int_M\left(S + \frac{1}{t}\right)\, dV_{g(t)} \le 0.
\end{equation}
This proves the lemma.
\end{proof}

In the sequel we will distinguish between the 
cases $\chi(M) > 0$, $\chi(M) = 0$ and $\chi(M) < 0$.

\subsection{Positive Euler characteristic} \label{subsec2.3}

\begin{prop} \label{prop2.5}
If $\chi(M) > 0$ then there is a finite singularity time $T < \infty$.
For any $p \in N$, the pointed
smooth limit
$\lim_{t \rightarrow T^-} \left( N,p, \frac{1}{T-t} h(t) \right)$
exists and is the isometric product of $\R$ with a sphere
$S^2$ of constant curvature $\frac12$.
\end{prop}
\begin{proof}
If a smooth flow existed for all $t \in [0, \infty)$ then
equation (\ref{2.26}) would imply that $V(t) < 0$ for large $t$,
which is impossible. Thus there is a singularity at some time
$T < \infty$.

From \cite[Theorem 5.15]{Li}, $\lim_{t \rightarrow T^-} \max_{m \in M} 
\left| R \right|(m,t) = \infty$, where $R$ is the scalar curvature
of $M$. Let $\{t_k\}_{k=1}^\infty$ be a sequence of times so that for 
sufficiently large $k$, 
$t_k$ is the first time $t$ for which $\max_{m \in M} |R|(m,t) = k$.
Let $m_k \in M$ be such that $|R|(m_k, t_k) = k$.
From \cite[Theorem 7.9]{Li}, a subsequence of 
the rescaled pointed solutions 
\begin{equation} \label{2.36}
(M, m_k, g_k(t), u_k(t)) = 
(M, m_k, k g(t_k + t/k), u(t_k + t/k))
\end{equation} 
converges smoothly to a solution $(M_\infty, m_\infty, g_\infty(\cdot),
u_\infty(\cdot))$ defined for $t \in (-\infty, 0]$, where
$g_\infty(\cdot)$ is a $\kappa$-solution on $M_\infty$
in the sense of \cite[Section 11.1]{Perelman} and
$u_\infty(\cdot)$ is constant both spatially and temporally.
The proof of this statement uses a modified $W$-functional
which was introduced in \cite{Li} and \cite{L}, and which
becomes
\begin{equation} \label{2.37}
W(g,u,f,\tau) = \int_{M} \left[ \tau \left( |\nabla f|^2 + R
- |\nabla u|^2 \right)
+f-2 \right] (4\pi\tau)^{-1} e^{-f} dV_{g(t)}.
\end{equation}
in our three-dimensional warped product case.

The only $\kappa$-solution on an orientable surface is the
round shrinking $2$-sphere \cite[Corollary 40.1]{KL}. 
In particular, $\lim_{k \rightarrow \infty} (M, k g(t_k), u(t_k))
= (S^2, g_{S^2}, u_\infty)$, where $g_{S^2}$ has constant scalar
curvature one.  (Because $S^2$ is compact, we no longer have
to refer to basepoints. At this point our convergence is still
modulo diffeomorphisms.) Standard arguments show that the
system (\ref{2.7}) is stable around the solution given by the
round shrinking $S^2$ and constant $u$, with exponential convergence
for the normalized flow.
It follows that $\lim_{t \rightarrow T^-} 
\left( M, \frac{1}{T-t} g(t), u(t) \right) = 
(S^2, g_{S^2}, u_\infty)$, where the convergence
is now taken without diffeomorphisms.

In terms of the three-dimensional geometry, from (\ref{2.18}) the
fiber lengths are uniformly bounded above and below by positive
constants, up to time $T$. The three-dimensional pointed limit 
$\lim_{t \rightarrow T^-} \left( N, \frac{1}{T-t} h(t), p \right)$
is the isometric product of $\R$ with $\left( S^2, g_{S^2} \right)$.
This proves the proposition.
\end{proof}

\begin{remark} \label{rem2.6}
We could also prove Proposition \ref{prop2.5} by
looking at the three-dimensional singularity models with
a nowhere-vanishing Killing vector field.  Such a proof would
be less elementary, since it would use the results of
\cite[Section 11]{Perelman}.
\end{remark}

\begin{remark} \label{rem2.7}
The method of proof of Proposition \ref{prop2.5} works for 
finite-time singularities of warped products $M \times S^1$
if $M$ is a compact manifold of arbitrary dimension $n-1$.
Blowing up at points of maximal curvature, one obtains an
$(n-1)$-dimensional ancient solution which is $\kappa$-noncollapsed
at all scales.
\end{remark}

\subsection{Nonsingularity when $\chi(M) \le 0$} \label{subsec2.4}

\begin{prop} \label{prop2.8}
If $\chi(M^2) \le 0$ then the Ricci flow exists for
$t \in [0, \infty)$.
\end{prop}
\begin{proof}
If not then there is a singularity at some time $T< \infty$.
The same argument as in Subsection \ref{subsec2.3} gives
$\lim_{k \rightarrow \infty} (M, k g(t_k), u(t_k))
= (S^2, g_{S^2}, u_\infty)$. In particular, $M$ is diffeomorphic
to $S^2$, which contradicts our assumption.
\end{proof}

\subsection{Vanishing Euler characteristic} \label{subsec2.5}

\begin{prop} \label{prop2.9}
If $\chi(M) = 0$ then $\lim_{t \rightarrow \infty} h(t)$
exists and is a flat metric on $T^3$. The convergence is
exponentially fast.
\end{prop}
\begin{proof}
We first show in the following two propositions that if $\chi(M) = 0$ then
$\max_{p \in N} \left| \rem^N \right| (p,t)  = O(t^{-1})$ and 
$\diam(N, h(t)) = O(\sqrt{t})$.
Here $\rem^N$ denotes the three-dimensional sectional curvatures.

\begin{prop} \label{prop2.10}
If $\chi(M) = 0$ then there is a $C <\infty$ so that for all
$t \in [0, \infty)$, we have 
\begin{equation} \label{2.38}
t\cdot \max_{p \in N} \left| \rem^N \right| (p,t) \le C.
\end{equation}
\end{prop}

\begin{proof}
If not, $\limsup_{t\to\infty} t\cdot \max_{m \in M} 
\left| \rem^N \right| (m,t) = \infty$. 
(Since the function $\left| \rem^N \right| (\cdot, t)$ pulls back 
from $M$, we can think of it as a function on $M$.)
We perform an analog of Hamilton's pointpicking algorithm for a
type IIb Ricci flow solution; see
\cite[Chapter 8.2.1.3]{Chow-Lu-Ni}.
Namely, take any sequence $\{T_i\}_{i=1}^\infty$ with
$\lim_{i \rightarrow \infty} T_i = \infty$ and 
let $(m_i, t_i) \in M \times
[0, T_i]$ be such that
\begin{equation} \label{2.39}
t_i (T_i - t_i) \left| \rem^N \right| (m_i,t_i) = 
\sup_{(m, t) \in M \times [0,T_i]} 
t(T_i - t) \left| \rem^N \right| (m,t).
\end{equation}
Dilate the flow in space and time by 
$Q_i = \left| \rem^N \right| (m_i,t_i)$ in the following way :
\begin{equation} \label{2.40}
g_i(t) = Q_i\cdot g(t_i+tQ_i^{-1}), \qquad u_i(t) = u(t_i + tQ_i^{-1}).
\end{equation}
Then $(g_i(\cdot), u_i(\cdot))$ satisfies (\ref{2.8}) on a time interval
$(A_i, \Omega_i)$, with $\lim_i A_i  = - \infty$ and 
$\lim_i \Omega_i = \infty$. Put
\begin{equation} \label{2.41}
h_i(t) = g_i(t) + e^{2u_i} d\theta^2.
\end{equation}
By construction,
$\left| \rem^N_{h_i} \right| (m_i, 0) = 1$ and 
after redefining $A_i$ and $\Omega_i$,
there is a sequence
$\{\gamma_i\}_{i=1}^\infty$ with $\lim_{i \rightarrow \infty} \gamma_i = 1$
such that 
\begin{equation} \label{2.42}
\max_{(m,t) \in M \times [A_i, \Omega_i]} 
\left| \rem^N_{h_i}  \right| (m, t)\le \gamma_i.
\end{equation}

The curvature bound on $N$ implies
a curvature bound on $M$ which, along with the {\em a priori} bounds on
$u_i$ and $|\nabla u_i|$, implies higher derivative bounds on
$\rem^M(g_i)$ and $u_i$ \cite[Theorem 5.12]{Li}. We would now like to
take a convergent subsequence of the pointed flows
$\{(M \times [A_i, \Omega_i], (m_i, 0), g_i(\cdot), u_i(\cdot))
\}_{i=1}^\infty$ 
to obtain an eternal solution
$(M_\infty \times \R, (m_\infty, 0), g_\infty(\cdot), u_\infty(\cdot))$ of
(\ref{2.8}),
where $\R$ denotes a time interval.
To do so, we need a uniform positive lower bound on the
injectivity radius at $(m_i, 0)$. If the two-dimensional manifolds
$(M, g_i(0))$
were positively curved
then such a bound would be automatic.  Since we don't know
that 
$(M, g_i(0))$
is positively curved, we argue differently.

\begin{lem} \label{lem2.11}
There is some $\epsilon > 0$ such that for all $i$, 
the injectivity radius of $(M, g_i(0))$ at $m_i$ is bounded
below by $\epsilon$.
\end{lem}
\begin{proof}
Suppose that the lemma is false.
Then after passsing to a subsequence, we can assume that
$\lim_{i \rightarrow \infty} \injrad_{g_i(0)} (m_i) = 0$.
Passing to a further subsequence, we can assume that
$\{(M \times [A_i, \Omega_i], (m_i, 0), g_i(\cdot), u_i(\cdot))\}_{i=1}^\infty$
converges to a solution of (\ref{2.8}) on an
\'etale groupoid.  For information about the use of \'etale
groupoids in Ricci flow, we refer to \cite{L0} and \cite{L}. The
upshot is that after passing to a subsequence, we have smooth
pointed convergence 
to an eternal solution 
$(M_\infty \times \R, (m_\infty, 0), g_\infty(\cdot), u_\infty(\cdot))$
of (\ref{2.8}), where $M_\infty$ is a two-dimensional \'etale groupoid.
Furthermore, this solution has uniformly bounded curvature by (\ref{2.42}),
and for each $t \in \R$, $(M_\infty, g_\infty(t))$ is a complete
closed effective Riemannian groupoid.

By assumption, $\lim_{i \rightarrow \infty} t_i Q_i = \infty$. 
As
\begin{equation} \label{2.43}
\max_{m \in M} {|\nabla u_i|^2_{g_i}(m,t)} = Q_i^{-1} 
\max_{m \in M} |\nabla u|^2_g
\left( m, t_i + t Q_i^{-1} \right),
\end{equation}
when combined with (\ref{2.19}) we conclude that
$\nabla u_\infty = 0$ on any time slice. Then $u_\infty$ is constant
spatially and temporally, and
$g_\infty$ is just a Ricci flow on 
$M_\infty$. As it is an eternal solution, it has nonnegative curvature,
with positive curvature at $(m_\infty, 0)$. 

We know that $(M_\infty, m_\infty, g_\infty(t))$ is a bounded curvature limit of
$\{(M, m_i, g_i(t))\}_{i=1}^\infty$. The Riemannian groupoid 
$(M_\infty, g_\infty(t))$ has a locally constant sheaf of finite
dimensional Lie algebras ${\mathfrak g}$ 
which act as germs of Killing vector fields
on the unit space $M_\infty^{(0)}$. Because of the bounded
curvature assumption, these
local Killing vector fields do not have a point of common vanishing;
see, for example, \cite[Theorem 5.1]{Rong}.

As $M_\infty^{(0)}$ is two-dimensional, and we are in the
collapsing situation, the only possibilities for
${\mathfrak g}$ are $\R^2$ and $\R$. If ${\mathfrak g} = \R^2$ then
$(M_\infty, g_\infty(t))$ is flat, which is a contradiction.

Suppose that ${\mathfrak g} = \R$. Locally, $g_\infty(t)$ can be
written as $dx^2 + f^2(x,t) dy^2$. Since its curvature is
$K = - \frac{f''}{f}$, the function $f$ is concave. Fixing
the value of $f$ at a single point in $M_\infty^{(0)}$,
the function $f$ pulls back from a function on the orbit
space ${\mathcal O}$ of $M_\infty$, 
which 
in our case is a one-dimensional orbifold
\cite[Proposition 5.2]{L}.
If ${\mathcal O}$ is
a circle then we immediately get a contradiction, since
the concave function $f$ must be constant, but this contradicts
the fact that $(M_\infty, g_\infty)$ has nonzero curvature at
$(m_\infty, 0)$.
If ${\mathcal O}$ is an interval orbifold then we can pass to a
double cover and argue as before. If ${\mathcal O}$ is $\R$ then
the positive concave function $f$ must be constant, which again
contradicts the fact that $(M_\infty, g_\infty)$ has nonzero curvature at
$(m_\infty, 0)$. 
If ${\mathcal O}$ is $[0, \infty)$ then we can pass to a double
cover and argue as before.

This proves the lemma.
\end{proof}

We can now take a convergent subsequence of
the pointed flows
$\{(M \times [A_i, \Omega_i], (m_i, 0), g_i(\cdot), u_i(\cdot))
\}_{i=1}^\infty$ 
to obtain an eternal solution
$(M_\infty \times \R, (m_\infty, 0), g_\infty(\cdot), u_\infty(\cdot))$ of
(\ref{2.8}). From (\ref{2.19}) and (\ref{2.43}), $u_\infty$ is constant.
Thus we have a nonflat 
eternal Ricci flow solution on the two-dimensional manifold
$M_\infty$, with uniformly bounded curvature 
(from (\ref{2.42})) and complete time
slices. In particular, 
$\int_{M_\infty} R_{g_\infty(0)} \, dV_{g_\infty(0)} > 0$.
Although we won't really need it,
$(M_\infty, g_\infty(\cdot))$ must be the cigar soliton,
as follows from the fact that the spacetime supremum of $R$ is 
achieved at a point, along with a differential Harnack inequality
\cite[Theorem 9.4]{Chow-Lu-Ni}.  Hence
\begin{equation} \label{2.44}
\int_{M_\infty} R_{g_\infty(0)} \, dV_{g_\infty(0)} = 4\pi.
\end{equation}

For large $i$, there is a bounded domain $S_i \subset (M, g(t_i))$ 
which, after rescaling, is
almost isometric to a large piece of the time-zero slice of the cigar soliton.
Then for large $i$, 
\begin{equation} \label{2.45}
\int_{S_i} R_{g(t_i)} \, dV_{g(t_i)} \ge 3\pi.
\end{equation}
On the other hand,
\begin{equation} \label{2.46}
\int_{M - S_i} R_{g(t_i)} \, dV_{g_i(t)} \ge - \frac{1}{t_i} \vol_{g(t_i)}(M - S_i) \ge
- \, \frac{V(t_i)}{t_i}.
\end{equation}
Adding (\ref{2.45}) and (\ref{2.46}), and using (\ref{2.27}), we see that for
large $i$,
\begin{equation} \label{2.47}
\int_M R_{g(t_i)} \, dV_{g(t_i)} \ge 2\pi.
\end{equation}
This contradicts the assumption that $\chi(M) = 0$ and proves
Proposition \ref{prop2.10}
\end{proof}

\begin{prop} \label{prop2.12}
$\diam(N,h(t)) = O( \sqrt{t})$.
\end{prop}
\begin{proof}

We will use the following general result about Riemannian submersions.
\begin{lem} \label{lem2.13}
If $\pi : (N,h) \rightarrow (M, g)$ is a Riemannian submersion, with $N$
compact and connected, then
\begin{equation} \label{2.48}
\diam(M,g) \le \diam(N,h) \le \diam(M,g) + 
2 \max_{m \in M} \diam(\pi^{-1}(m)),
\end{equation}
where $\diam(\pi^{-1}(m))$ is the intrinsic diameter of $\pi^{-1}(m)$.
\end{lem}
\begin{proof}

Given $m_1, m_2 \in M$, we have
\begin{equation} \label{2.49}
d_M(m_1, m_2) = d_N(\pi^{-1}(m_1), \pi^{-1}(m_2)).
\end{equation}
It immediately follows that $\diam(M, g) \le \diam(N, h)$.

Given $p_1, p_2 \in N$,
put $m_1 = \pi(p_1)$ and $m_2 = \pi(p_2)$.
Then 
\begin{equation} \label{2.50}
d_N(p_1, p_2) \le \diam(\pi^{-1}(m_1)) + d_N(\pi^{-1}(m_1), \pi^{-1}(m_2)) +
\diam(\pi^{-1}(m_2)).
\end{equation}
The lemma follows.
\end{proof}

From (\ref{2.18}) and Lemma \ref{lem2.13}, 
$\diam(N,h(t)) \le \diam(M,g(t)) + 2 e^{C_2}$. Thus it suffices to show
that $\diam(M,g(t)) = O(\sqrt{t})$.
 
Let $L(t)$ be the length (with respect to $g(t)$) of a shortest noncontractible closed geodesic $\gamma_t$ on $M$.
We parametrize $\gamma_t$ by an arclength parameter $s$.
We will first show the following claim, which is an analog of
\cite[Theorem 12.1]{Ha}.

\begin{claim} \label{claim2.14}
$L(t)$ is nondecreasing in $t$.
\end{claim}

\begin{proof}
Given the curve $\gamma_t$ at time $t > 0$, 
we obtain an upper bound on $L(t-\Delta t)$
by considering the length of 
the same curve $\gamma_t$ at the earlier time $t - \Delta t$. 
Then
\begin{align} \label{2.51}
\liminf_{\Delta t \rightarrow 0} \frac{L(t) - L(t-\Delta t)}{\Delta t} & \ge
\frac12 \int_0^{L(t)} 
\frac{\partial g}{\partial t}(\gamma_t', \gamma_t') \, ds \\
& = \int_0^{L(t)} \left[ - \frac12
R(\gamma_t(s)) + \langle \gamma_t', \nabla u \rangle^2 \right] \, ds.
\notag
\end{align}
As $\gamma_t$ is stable, applying the second variation formula with respect to a parallel
normal field along $\gamma_t$ gives
\begin{equation} \label{2.52}
- \int_0^{L(t)} R(\gamma_t(s)) \, ds \ge 0.
\end{equation}
The claim follows.
\end{proof}

Since $\chi(M) = 0$, equations (\ref{2.27}) and (\ref{2.38}) imply that
\begin{equation} \label{2.53}
\lim_{t \rightarrow \infty} \left[ \left( \max_{m \in M} |R|(m,t) \right)
\cdot V(t) \right] = 0.
\end{equation}

The next claim is purely geometric and has nothing to do with flows.

\begin{claim} \label{claim2.15}
There exist $\epsilon > 0$ and $C < \infty$ so that for every metric 
$g$ on $M = T^2$ with 
$\left( \max_{m \in M} |R_g|(m) \right) \cdot \vol(M,g) \le \epsilon$, we have 
\begin{equation} \label{2.54}
L(g) \cdot \diam(M,g) \le C\cdot \vol(M,g),
\end{equation}
where $L(g)$ is the length of the shortest noncontractible closed geodesic.
\end{claim}

\begin{proof}
Suppose first that the metric $g$ is flat. Let $\gamma$ be a shortest
closed geodesic and let $\widehat{M}$ be the corresponding cyclic cover
of $M$. Then $\widehat{M} = S^1 \times \R$, where the circle has length
$L(\gamma) = L(g)$.
A generator of the covering group $\Z$ acts on $\widehat{M}$ by
translation in the $\R$-direction, by some distance $L^\prime$, along
with rotation around $S^1$ by some angle $\theta \in [- \, \pi, \pi]$.
For any $\widehat{m} \in \widehat{M}$, we have
\begin{equation} \label{2.55}
d_{\widehat{M}}
( \widehat{m}, g \widehat{m}) = 
\sqrt{
(L^\prime)^2 + \left( \frac{\theta L}{2\pi} \right)^2
}.
\end{equation}
This is the length of a closed geodesic on $M$, so we must have
\begin{equation} \label{2.56}
\sqrt{(L^\prime)^2 + \left( \frac{\theta L}{2\pi} \right)^2} \ge L.
\end{equation}
This implies that $L^\prime \ge \frac{\sqrt{3}}{2} L$.
Now $\diam(M,g) \le \frac12 (L + L^\prime)$
and $\vol(M, g) = L L^\prime$, so the claim holds 
in this case with $C=2$.
Hence we can assume that $g$ is nonflat.

Suppose that the claim is false.  Then there is a sequence $\epsilon_i \rightarrow 0$ so that
for each $i$, there is a metric $g_i$ on $M$ with 
\begin{equation} \label{2.57}
\left( \max_{m \in M} |R_{g_i}|(m) \right) \cdot 
\vol(M, g_i) \le \epsilon_i
\end{equation}
but 
\begin{equation} \label{2.58}
L(g_i) \cdot \diam(M, g_i) > i \cdot \vol(M, g_i).
\end{equation}
Rescale $g_i$ so that $\max_{m \in M} |R_{g_i}|(m) = 1$. Then
$\vol(M, g_i) \le \epsilon_i$.

Let $\delta << 1$ be a new parameter.  Suppose first that for an infinite number of $i$,
the pointed Gromov-Hausdorff distance from $\left( M, m, g_i \right)$ to a point is
greater than $\delta$ for all $m \in M$, i.e.
$\diam(M, g_i) > \delta$. After relabelling the sequence 
$\{(M, g_i)\}_{i=1}^\infty$, we can assume that for all $i$ and all $m \in M$, 
the pointed Gromov-Hausdorff distance from $\left( M, m, g_i \right)$ to a point is
greater than $\delta$. Since $\lim_{i \rightarrow \infty} \vol(M, g_i) = 0$,
there is a sequence $\{c_i\}_{i=1}^\infty$ with 
$\lim_{i \rightarrow \infty} c_i = 0$ so that for any $m \in M$, the pointed closed metric
ball $\left( \overline{B \left( m, \frac{\delta}{10} \right)}, 
m, g_i \right)$ is $c_i$-close
to $\left( \left[ - \frac{\delta}{10}, \frac{\delta}{10} \right], 0 \right)$
in the pointed Gromov-Hausdorff topology.

Fix $\delta^\prime << 1$.
From \cite{Cheeger-Fukaya-Gromov}, for large $i$ there is a diffeomorphism
$\phi_i$ from $(M, g_i)$ to $(S^1 \times S^1, g^\prime_i)$, where
\begin{itemize}
\item $g^\prime_i$ is a warped product metric $dx^2 + f_i^2(x) dy^2$,
\item $\max_{x \in S^1} f_i(x) \rightarrow 0$ as $i \rightarrow \infty$, and 
\item $\phi_i$ is a $e^{\delta^\prime}$-biLipschitz map.
\end{itemize}
More precisely, we are using the fact that the results of 
\cite{Cheeger-Fukaya-Gromov} hold in a localized sense, i.e.
without an upper diameter bound; see
\cite[Section 2]{Cheeger-Tian} for discussion. The paper
\cite{Cheeger-Fukaya-Gromov} gives a biLipschitz approximation of
$(M, g_i)$ by a Riemannian nilbundle with affine holonomy; see also
\cite{Fukaya}. In
the present case such a Riemannian nilbundle is a Riemannian
submersion $S^1 \times S^1 \rightarrow S^1$.

For the metric $g^\prime_i$, let $A_i$ be the length of the circle base.
We know that $A_i \ge \frac{\delta}{10}$. 
To get an upper bound on $L(g_i)$, we locate a shortest circle fiber of
the fiber bundle $S^1 \times S^1 \rightarrow S^1$
and take its preimage under $\phi_i$. This gives a 
noncontractible closed curve in $M$, so
\begin{equation} \label{2.59}
L(g_i) \le e^{\delta^\prime} \min_{x \in S^1} f_i(x).
\end{equation}
Next, using Lemma \ref{lem2.13},
 the diameter of $(S^1 \times S^1, g^\prime_i)$ is bounded above by
$A_i + 2 \max_{x \in S^1} f_i(x)$, so
\begin{equation} \label{2.60}
\diam(M, g_i) \le e^{\delta^\prime}(A_i + 2 \max_{x \in S^1} f_i(x)).
\end{equation}
The volume of $(S^1 \times S^1, g^\prime_i)$ is bounded below by
$A_i$ times the length of the smallest circle fiber, so
\begin{equation} \label{2.61}
\vol(M, g_i) \ge e^{-2 \delta^\prime} A_i \min_{x \in S^1} f_i(x).
\end{equation}
For large $i$, equations (\ref{2.59}), (\ref{2.60}) and (\ref{2.61}) contradict
(\ref{2.58}).

Thus we can assume that for all but a finite number of $i$,
there is some point $m_i$ so that the pointed Gromov-Hausdorff
distance from $(M, m_i, g_i)$ to a point is at most $\delta$.
We now rescale $(M, g_i)$ to a metric $(M, \widehat{g}_i)$ with 
diameter one.  After this rescaling,
$\max_{m \in M} \left| R_{\widehat{g}_i} \right|(m) \le \delta^2$. 
Suppose that $\liminf_{i \rightarrow \infty} \vol(M, \widehat{g}_i) = 0$.
After passing to a subsequence, we can assume that
$\lim_{i \rightarrow \infty} \vol(M, \widehat{g}_i) = 0$. As in the
argument of the preceding paragraphs, for large $i$ there is a
biLipschitz approximation of $(M, \widehat{g}_i)$ by a warped product
metric on a fiber bundle $S^1 \times S^1 \rightarrow S^1$,
whose base has diameter close to one, and we obtain a contradiction to 
(\ref{2.58}). 

Thus there is some $v_0 > 0$ so that for all $i$, we have 
$\vol(M, \widehat{g}_i) \ge v_0$.
Cheeger compactness now gives a contradiction to (\ref{2.58}). This proves
the claim.
\end{proof}

To conclude the proof of Proposition \ref{prop2.12} we argue as follows.  
By (\ref{2.27}), $V(t) = O(\sqrt{t})$. By (\ref{2.2}) and
Proposition \ref{prop2.10},
$\max_{m \in M} |R^M|(m,t) = O(t^{-1})$. Thus for large $t$,
we can apply Claim \ref{claim2.15} to conclude that
\begin{equation} \label{2.62}
\diam(M, g(t)) \le C \frac{V(t)}{L(t)} \le C \frac{V(t)}{L(0)},
\end{equation}
where Claim \ref{claim2.14} is used in the last inequality.  Proposition 
\ref{prop2.12} follows.
\end{proof}

By \cite[Theorem 1.2.1]{L}, Propositions \ref{prop2.10} and 
\ref{prop2.12} imply that
$\max_{m \in M} |\rem^N_n|(m,t) = o(t^{-1})$.
Thus 
\begin{equation} \label{2.63}
\lim_{t \rightarrow \infty}
\left[
\left( \max_{m \in M} |\rem^N|(m,t) \right)
\cdot \diam^2(N, h(t)) \right] = 0.
\end{equation}

For some given $t$, 
we rescale $(N, h(t))$ so that it has diameter one.
We now wish to apply the local stability of flat metrics on $T^3$ under
the Ricci flow.  For this, we need the following lemma, which is 
purely geometric.

\begin{claim} \label{claim2.16}
There are a compact subset $K$ of the moduli
space of flat metrics on $T^3$ and
a function $\epsilon^\prime : [0, \infty) \rightarrow
[0, \infty)$, with $\lim_{\epsilon \rightarrow 0} \epsilon'(\epsilon) = 0$, 
so that the following holds.  Suppose that
$h$ is a Riemannian metric on $N = T^3$ with $\diam(N, h) = 1$ and
$\max_{p \in N} |\rem_h|(p) \le \epsilon$. Then there is a finite cover
$\widehat{N}$ of $N$ so that $(\widehat{N}, \widehat{h})$ has 
distance at most $\epsilon^\prime(\epsilon)$ in the $C^1$-topology from 
an element of $K$.
\end{claim}
\begin{proof}
Given $i_0, D_1 > 0$ and $D_2 < \infty$, let $K_{i_0,D_1, D_2}$ 
denote the isometry classes of
flat Riemannian metrics on $T^3$ with diameter in $[D_1, D_2]$ and injectivity
radius bounded below by $i_0$. For any $\alpha \in \Z^+$, the set
$K_{i_0, D_1, D_2}$ is compact in the $C^\alpha$-topology. We will take
$K = K_{i_0, D_1, D_2}$, where the parameters $i_0$, $D_1$ and $D_2$ will be
determined in the proof.

Given $K$, suppose that the claim is false. Then there is 
some $\epsilon^\prime > 0$ along with a
sequence $\{(N_i, h_i)\}_{i=1}^\infty$ of Riemannian $3$-tori
with $\diam(N_i, h_i) = 1$ and
$\max_{p \in N_i} |\rem_{g_i}|(p) \le \frac{1}{i}$, but with
the property that for no $i$ is there a finite cover
$\left( \widehat{N}_i, \widehat{h}_i \right)$ of $(N_i, h_i)$ 
with distance less than $\epsilon^\prime$ in the $C^1$-topology
from an element of $K$.
Let $\widetilde{N}_i$
be the universal cover of $N_i$.
Pick $p_i \in N_i$ and let $\widetilde{p}_i$
be a lift to $\widetilde{N}_i$. 
As $i \rightarrow \infty$,
$(\widetilde{N}_i, \widetilde{p}_i)$ converges to
$(\R^3, 0)$ in the pointed $C^1$-topology; this follows from
a uniform lower bound on the injectivity radius at
$\widetilde{p}_i$ \cite[(3.2)]{Rong} and $C^1$-convergence
results. Let $r_{1,i} \in
\pi_1(N_i, p_i)$ be a nontrivial element which minimizes
$d(\widetilde{p}_i, r_{1,i} \widetilde{p}_i)$. 
Let $r_{2,i} \in \pi_1(N_i, p_i) - \langle r_{1,i}
\rangle$ minimize $d(\widetilde{p}_i, r_{2,i} \widetilde{p}_i)$.
Let $r_{3,i} \in \pi_1(N_i, p_i) - \langle r_{1,i}, r_{2,i}
\rangle$ minimize $d(\widetilde{p}_i, r_{3,i} \widetilde{p}_i)$.
By the diameter condition, for large $i$ we have
$d(\widetilde{p}_i, r_{3,i} \widetilde{p}_i)
\le 10$. 

Let $\gamma_{1,i}, \gamma_{2,i}, \gamma_{3,i}$ be minimizing
geodesics from $\widetilde{p}$ to $r_{1,i} \widetilde{p},
r_{2,i} \widetilde{p}, r_{3,i} \widetilde{p}$.
There is a universal constant $c > 0$ so that for each $i$,
pairs from
$\{\gamma^\prime_{1,i}(0), \gamma^\prime_{2,i}(0), \gamma^\prime_{3,i}(0) \}$
form an angle at $\widetilde{p}$ which is at least $c$
\cite[Section 2.3]{Gromov}.
Let $k_{1,i}$ and $k_{2,i}$ be positive integers so that
$k_{1,i} d(\widetilde{p}_i, r_{1,i} \widetilde{p}_i)$ and
$k_{2,i} d(\widetilde{p}_2, r_{2,i} \widetilde{p}_i)$ lie in
$\left[ \frac{1}{10}, 10 \right]$. 
Let $\widehat{N}_i$ be the cover of $N_i$ so that
$\pi_1(\widehat{N}_i, \widehat{p}_i)$ has generators
$r_{1,i}^{k_{1,i}}, r_{2,i}^{k_{2,i}}, r_{3,i}$.
After passing to
a subsequence, we can assume that as $i \rightarrow \infty$, with
respect to the approximations of $\widetilde{N}_i$ by $\R^3$, the
triples
$r_{1,i}^{k_{1,i}} \widetilde{p}_i, r_{2,i}^{k_{2,i}}\widetilde{p}_i, 
r_{3,i}\widetilde{p}_i$
converge to vectors $v_1, v_2, v_3$ in $\R^3$ with
lengths in $\left[ \frac{1}{10}, 10 \right]$ and
mutual angles at least $c$. Then 
$\lim_{i \rightarrow \infty}
\left( \widehat{N}_i, \widehat{h}_i \right) = \R^3/(\Z v_1 + \Z v_2 + \Z v_3)$
in the $C^1$-topology. We can find $i_0, D_1 > 0$ and $D_2 < \infty$, 
computed in terms of
$c$, so that $\R^3/(\Z v_1 + \Z v_2 + \Z v_3)$ has diameter in 
$[D_1, D_2]$ and injectivity
radius bounded below by $i_0$. This contradicts the 
properties of the sequence $\{(N_i, h_i)\}_{i=1}^\infty$.
The claim follows.
\end{proof}

By (\ref{2.63}), for any $\epsilon^\prime > 0$, if $t$ is large enough then we
can apply Claim \ref{claim2.16} to $(N, h(t))$. 
Using the higher derivative
curvature estimates coming from the Ricci flow and applying a
similar argument as in the proof of Claim \ref{claim2.16}, 
for any $\alpha \in \Z^+$
we can also say that
there is a compact subset 
$K_\alpha$
of the moduli space of
flat metrics on $T^3$ so that for any $\epsilon^\prime > 0$,
if $t$ is large enough then a finite cover 
$\left( \widehat{N}, \widehat{h}(t) \right)$ is $\epsilon^\prime$-close
in the 
$C^\alpha$-topology
to an element of 
$K_\alpha$. Taking $\alpha$ large,
we can now apply
\cite[Theorem 3.7]{GIK} to conclude that as $t \rightarrow \infty$,
$\left( \widehat{N}, \widehat{h}(t) \right)$
(or more precisely the solution of (\ref{2.7})) 
converges to a flat metric on $T^3$. Furthermore, the convergence is
exponentially fast.
(Strictly speaking, \cite{GIK} considers the evolution of a Riemannian
metric $h$ which is sufficiently close to a fixed flat metric,
but the arguments clearly extend to the setting of closeness to
a compact set 
$K_\alpha$
of flat metrics, since there will be uniform
control on the constants.) Because of the equivariance
of Ricci flow under isometries, the same is true for
$(N, h(t))$.

This proves Proposition \ref{prop2.9}.
\end{proof}

\subsection{Negative Euler characteristic} \label{subsec2.6}

By Proposition \ref{prop2.8}, the flow (\ref{2.8}) or, equivalently, (\ref{2.7}) exists forever.
We start with the following proposition.

\begin{prop} \label{prop2.17}
We have $\max_{p \in N} |\rem^N|(p,t) = O(t^{-1})$. 
\end{prop}

\begin{proof}
The argument is similar to that in the proof of Proposition \ref{prop2.10}. 
If the statement were not true,  i.e. if 
$\limsup_{t\to\infty} t\cdot\max_{m \in M} |\rem^N|(m,t)  = \infty$, 
then we begin the argument exactly the same as in the proof of 
Proposition \ref{prop2.10} to take a limit of rescalings at times $\{t_i\}_{i=1}^\infty$.
This limit is an eternal solution of the form 
$h_{\infty}(t) = g_{\infty}(t) + C d\theta^2$, where 
$C$ is a constant and $g_{\infty}(t)$ is a nonflat eternal solution to the 
Ricci flow on a two-dimensional manifold $M_{\infty}$, with
complete time slices and nonnegative bounded curvature. 
As in the proof of Proposition \ref{prop2.10}, $(M_\infty, g_\infty(\cdot))$
must be a cigar soliton and
\begin{equation} \label{2.64}
\int_{M_{\infty}} R_{g_\infty(0)}\, dV_{g_\infty(0)} = 4\pi.
\end{equation}

For large $i$, there is a bounded domain $S_i \subset 
\left( M, g(t_i) \right)$ which, after rescaling, is
almost isometric to a large region in the cigar soliton.
Then 
\begin{equation} \label{2.65}
\int_{S_i} R_{g(t_i)}\, dV_{g(t_i)} \ge 3\pi.
\end{equation}

Furthermore, by (\ref{2.21}) and Corollary \ref{cor2.3}, we have
\begin{equation} \label{2.66}
\int_{M - S_i} R_{g(t_i)} \, dV_{g(t_i)} \ge
- \frac{1}{t_i} \vol_{g(t_i)}(M - S_i) \ge
- \frac{V(t_i)}{t_i} \ge 4 \pi \chi(M) \cdot (1 + o(i^0)).
\end{equation}
Adding (\ref{2.65}) and (\ref{2.66}), and taking $i \rightarrow \infty$,
contradicts the Gauss-Bonnet theorem for $M$.
\end{proof}

Put $\widehat{g}(t) = \frac{g(t)}{t}$. Let $\widehat{\nabla}$ denote
the corresponding Levi-Civita connection.

\begin{claim} \label{claim2.18}
There exist  $i_0 > 0$  and $t_0 > 0$ so that for every $t\ge t_0$, 
there is a point $m_t\in M$ where the injectivity radius
satisfies
$\injrad_{\widehat{g}(t)} (m_t) \ge i_0$.
\end{claim}

\begin{proof}
By (\ref{2.2}) and
Proposition \ref{prop2.17}, for $t \ge 1$ the metrics $(M, \widehat{g}(t))$
have uniformly bounded curvature.
If the claim were not true then for every $\epsilon > 0$, there
would be some $t_\epsilon \ge 1$ 
so that $\injrad_{\widehat{g}(t_\epsilon)}(m) < \epsilon$
for all $m \in M$. Then $M$ would have an F-structure and hence
a vanishing Euler characteristic
\cite{Cheeger-Gromov}. This is a contradiction.
\end{proof}

\begin{prop} \label{prop2.19}
For any $i_0 > 0$, 
define the $i_0$-thick part of $(M, \widehat{g}(t))$ by
\begin{equation} \label{2.67}
X_{i_0}(t) = 
\{m \in M \, : \, \injrad_{\widehat{g}(t)}(m) \ge i_0\}.
\end{equation}
Then
\begin{equation} \label{2.68}
\lim_{t \rightarrow \infty} \sup_{x \in X_{i_0}(t)} |R_{\widehat{g}(t)}(x) + 1|
= 0
\end{equation}
and
\begin{equation} \label{2.69}
\lim_{t \rightarrow \infty} \sup_{x \in X_{i_0}(t)} 
|\widehat{\nabla} u|_{\widehat{g}(t)}(x)
= 0.
\end{equation}
\end{prop}
\begin{proof}
Suppose that the proposition is not true.  Then there are some
$i_0, \epsilon > 0$ along with sequences $t_i \rightarrow \infty$
and $\{m_i\}_{i=1}^\infty$ so that for each $i$,
$\injrad_{\widehat{g}(t_i)}(m_i) \ge i_0$ and
either
$|R_{\widehat{g}(t_i)}(m_i) + 1| \ge \epsilon$ or
$|\widehat{\nabla} u|_{\widehat{g}(t_i)}(m_i) \ge \epsilon$.

By Hamilton's compactness theorem and the derivative estimates
on $u$ from \cite[Section 5]{Li}, after passing to a subsequence
we can assume that there is a smooth pointed limit of flows 
\begin{equation} \label{2.70}
\lim_{i \rightarrow  \infty} \left( M, m_i, \frac{1}{t_i} 
g(t_i t), u(t_i t) \right) =
\left( M_\infty, m_\infty, 
g_\infty(t), u_\infty(t) \right).
\end{equation}
For any bounded domain $S \subset M_\infty$, using Corollary \ref{cor2.2} we have
that at any time $a$,
\begin{equation} \label{2.71}
\int_S |\nabla u_\infty|_{g_\infty(a)}^2 \, dV_{g_\infty(a)} \le
\limsup_{i \rightarrow \infty} \int_M 
|\nabla u(t_i a)|_{g(t_i a)}^2 \, dV_{g(t_i a)} = 0.
\end{equation}
Thus $u_\infty$ is spatially constant at time $a$.
Because $a$ is arbitrary, and $u_\infty$ satisfies the time-dependent
heat equation,
it follows that $u_\infty$ is also temporally constant. In particular,
\begin{equation} \label{2.72}
0 = |\nabla u_\infty|_{g_\infty(1)}(m_\infty) = 
\lim_{i \rightarrow \infty} |\widehat{\nabla} u|_{\widehat{g}(t_i)}(m_i).
\end{equation}

Now $R_{g_\infty(t)} + \frac{1}{t} \ge 0$.
Given $0 < a < b < \infty$, equations (\ref{2.35}) and (\ref{2.72}),
along with Corollary \ref{cor2.3}, give
\begin{align} \label{2.73}
& \int_a^b \int_{M_\infty} \left( R_{g_\infty(t)} + \frac{1}{t} \right)
dV_{g_\infty(t)} \, \frac{dt}{t}  = \\
& \int_a^b \int_{M_\infty} \left( R_{g_\infty(t)} - |\nabla u_\infty|^2(t) 
+ \frac{1}{t} \right)
dV_{g_\infty(t)} \, \frac{dt}{t} \notag \\
& \le
\lim_{i \rightarrow \infty}
\int_{t_i a}^{t_i b} \int_{M} \left( R_{g(t)} - |\nabla u|^2(t) + 
\frac{1}{t} \right)
dV_{g(t)} \, \frac{dt}{t} \notag \\
& = \lim_{i \rightarrow \infty} \left(
\frac{V(t_i a)}{t_i a} - \frac{V(t_i b)}{t_i b}
\right) = 0. \notag
\end{align}
Since $a$ and $b$ were arbitrary, we obtain
$R_{g_\infty(t)}(m) = - \frac{1}{t}$ for all $m \in M_\infty$ and
$t \in (0, \infty)$.
In particular,
\begin{equation} \label{2.74}
-1 = R_{g_\infty(1)}(m_\infty) = \lim_{i \rightarrow \infty} 
R_{\widehat{g}(t_i)}(m_i).
\end{equation}
Equations (\ref{2.72}) and (\ref{2.74}) together contradict our assumptions
about $\{t_i\}_{i=1}^\infty$ and $\{m_i\}_{i=1}^\infty$, thereby
proving the proposition.
\end{proof}

\begin{remark} \label{rem2.20}
In the case $\chi(M) < 0$ one could hope for a bound
$\diam(M, g(t)) = O(\sqrt{t})$. With such a bound
one could conclude that in Proposition \ref{prop2.19}, there is some
$i_0 > 0$ so that for large $t$, $X_{i_0}(t)$ is all of $M$.
Without such a bound, one could imagine that as
$t \rightarrow \infty$, the manifolds $(M, \widehat{g}(t))$
approach a family of surfaces of constant curvature
$- \frac12$ that slowly pinch off a closed geodesic.
\end{remark}

\section{Torus bundles} \label{sec3}

In this section we prove Theorem \ref{thm1.6}. In Subsection \ref{subsec3.1}
we write down the Ricci flow equations with a $U(1) \times U(1)$
symmetry and give some direct consequences.  In Subsection
\ref{subsec3.2} we show that the Ricci flow exists for all $t \in [0, \infty)$.
In Subsection \ref{subsec3.3} we prove that the curvature decays like
$O \left( t^{-1} \right)$. In Subsection \ref{subsec3.4} we show that
the length of the circle base is $O(\sqrt{t})$. In Subsection
\ref{subsec3.5} we finish the proof of Theorem \ref{thm1.6}.

\subsection{Twisted principal $U(1) \times U(1)$ bundles} \label{subsec3.1}

Let $N$ be a an 
orientable $3$-manifold which is the total space of
a fiber bundle $\pi : N \rightarrow S^1$, with
$T^2$-fibers.
Choosing an orientation of $S^1$, the
fiber bundle has a
holonomy $H \in \SL(2, \Z) = \pi_0(\Diff^+(T^2))$. Taken up
to inverses, $H$ determines the topological type of the
fiber bundle.
We refer to \cite[Theorem 5.5]{Scott} for the Thurston types of
such fiber bundles.
If $H$ is elliptic, i.e. has finite order, then $N$
has a flat structure. If $H$ is parabolic,
i.e. $|\Tr(H)| \le 2$ but $H$ is not elliptic, then
$N$ has a $\Nil$ structure.  If $H$ is hyperbolic, i.e. has
no eigenvalues on the unit circle, then $N$ has a $\Sol$ structure.

If $H$ is the identity then $N = S^1 \times T^2$ is the total space
of a principal $U(1) \times U(1)$ bundle.  That is, $N$ admits a
free $U(1) \times U(1)$ action.  In general, $N$ is the total
space of a twisted principal $U(1) \times U(1)$ bundle, where
``twisted''
refers to the fact that $H$ may be nontrivial.
The setup is a special case of that in
\cite[Section 4.1]{L}.
Let ${\mathcal E}$ be a local system
over $S^1$ of groups isomorphic to $U(1) \times U(1)$. We 
assume that the holonomy of the local system is 
$H \in \Aut(U(1) \times U(1))$. Then there is a notion of a
free ${\mathcal E}$-action on the $T^2$-bundle $N$, 
which generalizes the global
$U(1) \times U(1)$ action that exists when $H$ is the identity.

Let $h$ be a Riemannian metric on $N$ which is ${\mathcal E}$-invariant.
There is a corresponding horizontal distribution ${\mathcal H}$ on $N$. Since
${\mathcal H}$ is one-dimensional, it is
integrable. Note that even if $H$ is the identity,
${\mathcal H}$ can have a nontrivial holonomy in $U(1) \times U(1)$,
when going around the circle base.
Thus the flat structure on ${\mathcal E}$ is logically distinct
from the flat structure on $N$ coming from ${\mathcal H}$. 

Let $V$ be a coordinate chart of $S^1$, with local coordinate $y$.
The integrability of ${\mathcal H}$ gives a local trivialization
$V \times T^2$ of $\pi^{-1}(V)$. The restriction of
$h$ to $\pi^{-1}(V)$ is invariant under the $U(1) \times U(1)$ action
coming from ${\mathcal E} \big|_V$.
Using this action and the local trivialization, 
let $x_1, x_2$ denote local angular coordinates on
the $T^2$-fibers.
In terms of these coordinates we can write
\begin{equation} \label{3.1}
h = \sum_{i,j = 1}^2 G_{ij}(y) dx^i dx^j + g_{yy}(y) \, dy^2.
\end{equation} 

We use $i,j,k,l$ for vertical indices.
From \cite[Section 4.2]{L},
nonzero components of the curvature tensor of $(N,h)$ are
\begin{align} \label{3.2}
R^N_{ijkl} & = - \, \frac14 g^{yy} G_{ik,y} G_{jl,y} + \frac14
g^{yy} G_{il,y} G_{jk,y}, \\
R^N_{iyjy} & = - \, \frac12 G_{ij;yy} + \frac14 G^{kl} G_{ik,y} G_{jl,y},
\notag
\end{align}
where 
\begin{equation}
G_{ij;yy} = G_{ij,yy} - \Gamma^y_{\: yy} G_{ij,y} =
G_{ij,yy} - \frac12 \, \frac{g_{yy,y}}{g_{yy}} \, G_{ij,y}.
\end{equation}

The nonzero components of the Ricci tensor are
\begin{align} \label{3.3}
R_{ij} & = - \, \frac12 g^{yy} G_{ij;yy} - \, \frac14   
g^{yy} G^{kl} G_{kl,y} G_{ij,y} + g^{yy} \frac12 G^{kl} G_{ik,y} G_{lj,y}, \\
R_{yy} & = - \, \frac12 G^{ij} G_{ij;yy} + \frac14 G^{ij} G_{jk,y} 
G^{kl} G_{li,y}. \notag
\end{align}
The scalar curvature is
\begin{equation} \label{3.4}
R = - \, g^{yy} G^{ij} G_{ij;yy} + \frac34 g^{yy} G^{ij} G_{jk,y}
G^{kl} G_{li,y} - \, \frac14 g^{yy} G^{ij} G_{ij,y} G^{kl} G_{kl,y}.
\end{equation}

We will use matrix notation $G = \left( G_{ij} \right)$.
Note that $G$ is symmetric and positive-definite.
The Ricci flow equation
\begin{equation} \label{3.5}
\frac{dh}{dt} = - \, 2 \, \ric_{h(t)}
\end{equation}
preserves the local $U(1) \times U(1)$ invariance of the metric.
In terms of $g$ and $G$, it becomes
\begin{align} \label{3.6}
\frac{\partial g_{yy}}{\partial t} & = \Tr \left( G^{-1} G_{;yy} \right)
- \, \frac12 \Tr \left( \left( G^{-1} G_{,y} \right)^2 \right), \\
\frac{\partial G}{\partial t} & = g^{yy} G_{;yy} + \frac12 g^{yy}
\Tr \left( G^{-1} G_{,y} \right) G_{,y} - g^{yy} G_{,y} G^{-1} G_{,y}. \notag
\end{align}
Adding a Lie derivative with respect to $- \nabla \ln \sqrt{\det(G)}$
to the right-hand side gives the modified equations
\begin{align} \label{3.7}
\frac{\partial g_{yy}}{\partial t} & = 
\frac12 \Tr \left( \left( G^{-1} G_{,y} \right)^2 \right), \\
\frac{\partial G}{\partial t} & = g^{yy} \left( G_{;yy} - G_{,y} G^{-1} G_{,y} \right).
\notag
\end{align}
Taking $y$ to run over $\R$, the periodicity condition on $G$ is
\begin{equation} \label{3.8}
G(y+1,t) = H^T G(y,t) H. 
\end{equation}
The manifold $N$ can be recovered by taking the quotient of
$\R \times T^2$ by the equivalence relation 
$(y+1, x) \sim (y, Hx + b)$, where $b$ is
some
fixed element of $\R^2/\Z^2$.

Hereafter we will mainly work with (\ref{3.7}).
An example of a solution to (\ref{3.7}) is
\begin{align} \label{3.9}
g_{yy}(y,t) & = 4c^2(t+a), \\
G(y,t) & = 
\begin{pmatrix}
e^{2cy} & 0 \\
0 & e^{-2cy}
\end{pmatrix}. \notag
\end{align}
If we take $y$ to be defined in $\R/\Z$ then we get
a metric on a bundle with hyperbolic holonomy
$H \in \SL(2, \Z)$, provided that
$H$ has eigenvalues $e^{c}$ and $e^{-c}$.
Here $\partial_{x^1}$ and $\partial_{x^2}$ are corresponding eigenvectors.
The length of the circle base is $2c\sqrt{t+a}$.

Given $s > 0$ and a solution $(g(\cdot), G(\cdot))$ of (\ref{3.7}), we
obtain another solution of (\ref{3.7}) by putting
\begin{align} \label{3.10}
g_s(t) & = \frac{1}{s} g(st), \\
G_s(t) & = G(st). \notag
\end{align}
Put
\begin{equation} \label{3.11}
h_s = (dx)^T G_s dx + g_s.
\end{equation}
Then
\begin{equation} \label{3.12}
\left| \rem_{h_s} \right|^2 = s^2 \left| \rem_{h} \right|^2.
\end{equation}

\begin{lem} \label{lem3.1}
\begin{equation} \label{3.13}
\frac{\partial \ln \det(G)}{\partial t} = \triangle \ln \det(G).
\end{equation}
\end{lem}
\begin{proof}
We have
\begin{align}
\frac{\partial \ln \det(G)}{\partial t} & =
\Tr \left( G^{-1} \frac{\partial G}{\partial t} \right) \\
& = g^{yy} \Tr \left( G^{-1} (G_{;yy} - G_{,y} G^{-1} G_{,y}) \right). \notag
\end{align}
Passing to an arc-length parameter $s$ on $S^1$ at a given
time $t$, we obtain
\begin{equation}
\frac{\partial \ln \det(G)}{\partial t} =
\Tr \left( G^{-1} (G_{,ss} - G_{,s} G^{-1} G_{,s}) \right).
\end{equation}
On the other hand
\begin{align}
\triangle \ln \det(G) & = \frac{d^2}{ds^2} \ln \det(G) =
\frac{d}{ds} \Tr \left( G^{-1} G_{,s} \right) \\
& = \Tr \left( G^{-1} G_{,ss} - G^{-1} G_{,s} G^{-1} G_{,s} \right). \notag
\end{align}
This proves the lemma.
\end{proof}

Note that ${\det(G)}$ is globally defined on $S^1$ since
the holonomy lies in $\SL(2, \Z)$. It represents the squares of the volumes
of the fibers.

\begin{cor} \label{cor3.2}
There are constants $C_1, C_2 > 0$ so that for all
$y \in S^1$ and all times $t$ for which the flow exists.
\begin{equation} \label{3.14}
C_1 \le \sqrt{\det(G)}(y,t) \le C_2.
\end{equation}
\end{cor}
\begin{proof}
This follows from applying the maximum principle to Lemma \ref{lem3.1}.
\end{proof}

\begin{lem} \label{lem3.3}
Put
\begin{equation} \label{3.15}
{\mathcal E} = g^{yy} \Tr \left( \left( G^{-1} G_{,y} \right)^2 \right).
\end{equation}
Then
\begin{equation} \label{3.16}
\frac{\partial {\mathcal E}}{\partial t} = \triangle {\mathcal E} - 
\frac{{\mathcal E}^2}{2} - 
2 \, g^{yy} g^{yy}
\Tr \left( \left( G^{-1} G_{;yy} - \left( G^{-1} G_{,y} \right)^2 \right)^2 \right).
\end{equation}
\end{lem}
\begin{proof}
Differentiating (\ref{3.15}) with respect to $t$ gives
\begin{align} \label{neweqn}
\frac{\partial {\mathcal E}}{\partial t} \:  = & \: - \: \frac12
g^{yy} g^{yy} \left( \Tr \left( \left( G^{-1} G_{,y} \right)^2 \right) 
\right)^2 \\
& \: - \:  2 g^{yy} \: \Tr
\left( G^{-1} G_{,y} G^{-1} 
g^{yy} (G_{;yy} - G_{,y} G^{-1} G_{,y} ) G^{-1} G_{,y}
\right) \notag \\
& \: + \:  2 g^{yy} \: \Tr
\left( G^{-1} G_{,y} G^{-1} 
\left( g^{yy} (G_{;yy} - G_{,y} G^{-1} G_{,y} ) \right)_{,y}
\right). \notag
\end{align}
Switching to an arc-length parameter $s$ at a given time $t$, we obtain
\begin{align}
\frac{\partial {\mathcal E}}{\partial t} \:  = & \: - \: 
\frac{{\mathcal E}^2}{2}
\: - \:  2  \: \Tr
\left( G^{-1} G_{,s} G^{-1} 
(G_{,ss} - G_{,s} G^{-1} G_{,s} ) G^{-1} G_{,s}
\right)  \\
& \: + \:  2 \: \Tr
\left( G^{-1} G_{,s} G^{-1} 
\left( G_{,ss} - G_{,s} G^{-1} G_{,s} \right)_{,s}
\right). \notag
\end{align}
One now computes that
\begin{align}
\frac{\partial {\mathcal E}}{\partial t} \:  = \: & - \: 
\frac{{\mathcal E}^2}{2}
\: + \:  \left( \Tr
\left( \left( G^{-1} G_{,s} \right)^2 \right) \right)_{,ss} \\
& - \: 2 \: \Tr \left( \left( G^{-1} G_{,ss} - \left( G^{-1} G_{,s} \right)^2
\right)^2 \right). \notag
\end{align}
This proves the lemma.
\end{proof}

Note that 
\begin{equation}{\mathcal E} = g^{yy} \Tr \left( \left( 
G^{-\frac12} G_{,y} G^{-\frac12} \right)^2 \right)
\end{equation}
is nonnegative, since $G^{-\frac12} G_{,y} G^{-\frac12}$ is symmetric.

\begin{cor} \label{cor3.4}
For all $y \in S^1$ and all $t \ge 0$ for which the flow exists,
we have ${\mathcal E}(y,t) \le \frac{2}{t}$.
\end{cor}
\begin{proof}
We have 
\begin{align}
& \Tr \left( \left( G^{-1} G_{;yy} - \left( G^{-1} G_{,y} \right)^2 \right)^2 \right)
= \\
& \Tr \left( \left( G^{-\frac12} G_{;yy} G^{-\frac12} - \left( 
G^{-\frac12} G_{,y} G^{-\frac12} \right)^2 \right)^2 \right), \notag
\end{align}
which is nonnegative since $G^{-\frac12} G_{;yy} G^{-\frac12} - \left( 
G^{-\frac12} G_{,y} G^{-\frac12} \right)^2$ is symmetric.
The corollary now 
follows from applying the maximum principle to Lemma \ref{lem3.3}.
\end{proof}

\subsection{Nonsingularity of the flow} \label{subsec3.2}

\begin{prop} \label{prop3.5}
Given $h(0)$ as described in Subsection \ref{subsec3.1}, the flow
(\ref{3.7}) exists for $t \in [0, \infty)$.
\end{prop}
\begin{proof}
If not then there is a singularity at some time $T < \infty$.
Specializing the modified $W$-functional of
\cite[Definition 4.48]{L} to our case, it becomes 
\begin{align} \label{3.17}
& W(G,g,f,\tau) = \\
& \int_{S^1} \left[ \tau \left( |\nabla f|^2 - \,
\frac14 g^{yy} \Tr \left( \left( G^{-1} G_{,y} \right)^2 \right) \right)
+f-1 \right] (4\pi\tau)^{- \, \frac12} e^{-f} \sqrt{g_{yy}} dy. \notag
\end{align}
Using this modified $W$-functional, we can go through the
same steps as in the proof of \cite[Theorem 7.9]{Li}
to conclude that there is a blowup limit 
$\left( M_\infty, m_\infty, g_\infty(\cdot), G_\infty(\cdot) \right)$ where
\begin{itemize}
\item $G_\infty = \const$. and
\item $g_\infty$ is a nonflat Ricci flow solution on the $1$-manifold
$M_\infty$.
\end{itemize}
However, there is no such nonflat Ricci flow solution.
This proves the proposition.
\end{proof}

\subsection{Curvature bound} \label{subsec3.3}

\begin{prop} \label{prop3.6}
We have $\max_{p \in N} \left| \rem^N \right| (p,t) = O \left( t^{-1} \right)$.
\end{prop}
\begin{proof}
Suppose that the proposition is false.  We take a rescaling limit
as in the proof of Proposition \ref{prop2.10} to obtain a nonflat eternal
solution $(M_\infty, p_\infty, g_\infty(\cdot), G_\infty(\cdot))$ on a
one-dimensional \'etale groupoid $M_\infty$. 
(It will follow from Lemma \ref{lengthlem} that $M_\infty$ is a one-dimensional
manifold.)
From Corollary \ref{cor3.4},
$G_\infty$ is constant. Then $g_\infty$ is a nonflat Ricci flow
solution on a one-dimensional space, which is a
contradiction.
\end{proof}

\subsection{Diameter bound} \label{subsec3.4}

We compute how the length of the base circle varies with time.
\begin{lem} \label{lem3.7}
Put $L(t) = \int_{S^1} \sqrt{g_{yy}}(y) dy$. Then
\begin{equation} \label{3.18}
\frac{dL}{dt} = \frac14 \int_{S^1} {\mathcal E}(y,t) \sqrt{g_{yy}}(y) dy.
\end{equation}
\end{lem}
\begin{proof}
We have
\begin{equation} \label{3.19}
\frac{dL}{dt} = \frac12 \int_{S^1} g^{yy} \frac{\partial g_{yy}}{\partial t} 
\sqrt{g_{yy}}(y) dy
= \frac14 \int_{S^1} g^{yy} \Tr \left( \left( G^{-1} G_{,y} \right)^2 \right) 
\sqrt{g_{yy}}(y) dy.
\end{equation}
This proves the lemma.
\end{proof}

In particular, $L(t)$ is monotonically nondecreasing in $t$.

\begin{lem} \label{lem3.8}
$t^{- \, \frac12} L(t)$ is monotonically nonincreasing in $t$.
\end{lem}
\begin{proof}
From Corollary \ref{cor3.4} and Lemma \ref{lem3.7}, we derive
\begin{equation} \label{3.20}
\frac{dL}{dt} \le \frac{L(t)}{2t}. 
\end{equation}
The lemma follows.
\end{proof}

\subsection{Long-time behavior} \label{subsec3.5}

\begin{lem} \label{lem3.9}
Given $x \in T^2$ and linearly independent vectors
$v_1, v_2 \in T_x T^2$, there is a constant
$C(v_1, v_2) < \infty$ with the following property.  Suppose that
$g_{T^2}$ is a flat metric on $T^2$. Then
\begin{equation} \label{3.21}
\diam \left( T^2, g_{T^2} \right) \le C(v_1, v_2) 
\left( |v_1|_{g_{T^2}} + |v_2|_{g_{T^2}} \right).
\end{equation}
\end{lem}
\begin{proof}
Let ${\mathcal V}_1$ and ${\mathcal V}_2$ be the affine-parallel vector
fields on $T^2$ that extend $v_1$ and $v_2$, respectively.
Let $g_0$ be a fixed flat metric on $T^2$. There is some
$c < \infty$ so that any pair of points in $T^2$ can
be joined by flowing first in the ${\mathcal V}_1$-direction
for a length at most $c$ and then flowing in the
${\mathcal V}_2$-direction for a length at most $c$. 
With respect to $g_{T^2}$, the length of this path is
bounded above by $c \frac{|v_1|_{g_{T^2}}}{|v_1|_{g_0}} + c 
\frac{|v_2|_{g_{T^2}}}{|v_2|_{g_0}}$. The
lemma follows.
\end{proof}

We use the notation in the statement of Theorem \ref{thm1.6}.

\subsubsection{Elliptic holonomy} \label{subsubsec3.5.1}

If $H$ is elliptic then after pulling back the $T^2$-bundle from
a finite covering $S^1 \rightarrow S^1$, we can assume that
$H = I$. Then $G(y,t)$ is a globally defined matrix-valued function
of $y \in \R/\Z$ and $t \in [0, \infty)$.

\begin{lem} \label{lem3.10}
The intrinsic diameter of the $T^2$-fibers is uniformly bounded 
above in $t$.
\end{lem}
\begin{proof}
Given $v \in \R^2$, equation (\ref{3.7}) implies that
\begin{equation} \label{3.22}
\frac{\partial}{\partial t} \langle v, G v \rangle =
g^{yy} \langle v, G v \rangle_{;yy} - g^{yy}
\langle v, G_{,y} G^{-1} G_{,y} v \rangle.
\end{equation}
The maximum principle now implies that
$\langle v, G(y,t) v \rangle$ is uniformly bounded above for
$y \in S^1$ and $t \in [0, \infty)$. The lemma follows from
Lemma \ref{lem3.9}.
\end{proof}

\begin{prop} \label{prop3.11}
$\diam(N, h(t)) = O(\sqrt{t})$.
\end{prop}
\begin{proof}
This follows from Lemmas \ref{lem2.13}, \ref{lem3.8} and \ref{lem3.10}.
\end{proof}

From \cite[Theorem 1.2.1]{L}, 
\begin{equation} \label{3.23}
\left( \max_{p \in N} \left|
\rem^N_h \right| (p,t) \right) \cdot \diam^2(N,h(t)) = o(t).
\end{equation}
The argument of Subsection \ref{subsec2.5} now shows that
$\lim_{t \rightarrow \infty} h(t)$ exists and is a flat metric on $N$,
with the convergence being exponentially fast.

\subsubsection{Hyperbolic holonomy} \label{subsubsec3.5.3}

Suppose that $H$ is hyperbolic.

Let $P(2,\R)$ denote the positive-definite symmetric $2 \times 2$
matrices. Given $G \in P(2, \R)$ and symmetric $2 \times 2$ matrices
$\delta_1 G, \delta_2 G \in T_G P(2, \R)$, we define their
inner product by
\begin{equation}
\langle \delta_1 G, \delta_2 G \rangle \: = \: \frac12 \:
\Tr \left( G^{-1} (\delta_1 G)  G^{-1} (\delta_2 G) \right).
\end{equation}

Consider the map $\Phi : \R \times \SL(2,\R) \rightarrow P(2, \R)$ given by
$(u, M) \rightarrow e^u M^T M$. Identifying $\SO(2) \backslash \SL(2, \R)$ with
the hyperbolic space $H^2$, the map $\Phi$ passes to an isometry
$\R \times H^2 \rightarrow P(2, \R)$.

With respect to this isometry,
the action of $\SL(2, \Z)$ on $P(2, \R)$ (by
$(A, G) \rightarrow A^T G A$) becomes the product of the trivial
action of $\SL(2, \Z)$ on $\R$ with the isometric action of
$\SL(2, \Z)$ on $H^2$. 
Letting $\langle H \rangle$ denote the cyclic subgroup of
$\SL(2, \Z)$ generated by the holonomy $H$,
equation (\ref{3.8})
can be interpreted as saying that for each $t$, the function
$G(\cdot, y)$ describes a smooth map
$S^1 \rightarrow (\R \times H^2/\langle H \rangle)$ whose
homotopy class is specified by $H$. 

\begin{lem} \label{lengthlem}
There is a constant $c > 0$ so that for all $t \in [0, \infty)$,
$L(t) \ge c \sqrt{t}$.
\end{lem}
\begin{proof}
Lemma \ref{cor3.4} says that the Lipschitz constant of the map
$G(\cdot, t)$ from $[0,1]$ to $P(2, \R)$
is bounded above by $\frac{1}{\sqrt{t}}$. As $H$ is hyperbolic,
there is a minimal length $c > 0$ among all noncontractible closed curves in 
$H^2/\langle H \rangle$. Consequently, the distance 
between $G(0,t)$ and $G(1,t) = H^T G(0,t) H$ in 
$P(2, \R)$ is bounded below by $c$.
Thus $L(t) \ge c \sqrt{t}$.
\end{proof}

Recall the definitions of $g_s$ and $G_s$ from (\ref{3.10}).
By an appropriate $s$-dependent choice of basis for $\R^2$, we can
assume that $G_s(0,1) = I$. In making such a choice of
basis, we are ignoring the 
lattice structure that comes from writing $T^2$ as a quotient of $\R^2$. 
We are simply treating $G$ and $g$ as functions which satisfy
(\ref{3.7}) and (\ref{3.8}).

Let $X$ be the real symmetric matrix so that $e^X = H^T H$.

\begin{prop} \label{Wconv}
For any sequence $\{s_j\}_{j=1}^\infty$ going to infinity,
after passing to a subsequence and possibly reparametrizing $S^1$, we have
\begin{equation}
\lim_{j \rightarrow \infty} g_{s_j}(y,t) \: = \: \frac{t}{2} \: 
\Tr(X^2) \: dy^2
\end{equation}
and
\begin{equation}
\lim_{j \rightarrow \infty} G_{s_j}(y,t) = e^{yX},
\end{equation}
with smooth convergence on compact subsets of $S^1 \times [0, \infty)$.
\end{prop}
\begin{proof}
The proof is similar to that in \cite[Propositions 4.39 and 4.79]{L};
see also \cite[Theorem 1.3]{FIN}.

We first construct a positive solution 
$\widetilde{u}_\infty$ of the conjugate heat equation
\begin{equation} \label{ueqn}
\frac{\partial u}{\partial t} = - \triangle u - \frac14 g^{yy} 
\Tr \left( \left( G^{-1} G_{,y} \right)^2 \right) u
\end{equation}
that is defined for $t \in [0, \infty)$. To do so,
note that if $u$ is a solution to (\ref{ueqn}) then $\int_{S^1}
u \: \sqrt{g_{yy}} \: dy$ is constant in $t$. Let $\{t_j\}_{j=1}^\infty$
be a sequence of times going to infinity. Let $\widetilde{u}_j(\cdot)$
be a solution to (\ref{ueqn}) on the interval $[0, t_j]$ with initial
condition 
$\widetilde{u}_j(t_j) = \frac{1}{L(t_j)}$.
One shows that
one can extract a subsequence of the $\widetilde{u}_j$'s that converges
smoothly on compact subsets of $S^1 \times [0, \infty)$ to a positive solution 
$\widetilde{u}_\infty(\cdot)$ of
(\ref{ueqn}).

Define $\widetilde{f}_\infty(\cdot)$ by $\widetilde{u}_\infty(t) = 
(4 \pi t)^{- \: \frac{n}{2}} e^{- \: \widetilde{f}_\infty(t)}$. 
Put
\begin{align}
& W_+(G,g,\widetilde{f},t) = \\
& \int_{S^1} \left[ t \left( |\nabla \widetilde{f}|^2 - \,
\frac14 g^{yy} \Tr \left( \left( G^{-1} G_{,y} \right)^2 \right) \right)
-\widetilde{f}+1 \right] (4\pi t)^{- \, \frac12} e^{-\widetilde{f}} 
\sqrt{g_{yy}} dy. \notag
\end{align}
From \cite[Proposition 4.64]{L},
\begin{align} \label{Wmon}
& \frac{d}{dt} W_+(G(t),g(t),\widetilde{f}_\infty(t),t) = \\
&\frac{t}{2} \int_{S^1} 
\Tr \left( \left( G^{-1} (\triangle G - g^{yy} G_{,y} G^{-1} G_{,y} -
g^{yy} G_{,y} \widetilde{f}_{\infty,y}) \right)^2 \right) \: 
\widetilde{u}_\infty \:
\sqrt{g_{yy}} \: dy \: + \notag \\
& 2t \int_{S^1} g^{yy} g^{yy} \left( - \frac14 \Tr \left(
\left( G^{-1} G_{,y} \right)^2 \right) + 
\widetilde{f}_{;yy} + \frac{1}{2t} g_{yy}
\right)^2 \: \widetilde{u}_\infty \:
\sqrt{g_{yy}} \: dy. \notag
\end{align}
In particular,
$W_+(G(t),g(t),\widetilde{f}_\infty(t),t)$ 
is monotonically nondecreasing in $t$.
Put
\begin{equation}
W_\infty = \lim_{t \rightarrow \infty} W_+(G(t),g(t),
\widetilde{f}_\infty(t),t),
\end{equation}
which at the moment could be infinity.

Using the curvature bound from Proposition \ref{prop3.6} and the
diameter bounds from Lemmas \ref{lem3.8} and \ref{lengthlem},
one shows that after passing to a subsequence,
$\lim_{j \rightarrow \infty} (g_{s_j}(\cdot), G_{s_j}(\cdot))$ exists in
the topology of smooth convergence on compact subsets
of $S^1 \times [0, \infty)$, and equals
a solution $(g_\infty(\cdot), G_\infty(\cdot))$ of (\ref{3.7})
on a circle of time-$1$ length 
\begin{equation}
L_\infty = \lim_{t \rightarrow \infty}
\frac{L(t)}{\sqrt{t}}.
\end{equation}
(The notion of convergence allows
for $j$-dependent diffeomorphisms of $S^1$.) 

Put ${u}_j(t) = \widetilde{u}_\infty(t+s_j)$. 
After passing to a subsequence, we
can assume that $\lim_{j \rightarrow \infty} u_j(t) = u_\infty(t)$
for some solution $u_\infty(\cdot)$ to (\ref{ueqn}) 
(relative to $g_\infty(\cdot)$ and $G_\infty(\cdot)$), with
smooth convergence on compact subsets of $S^1 \times [0, \infty)$. Define
$f_\infty(\cdot)$ by
${u}_\infty(t) = (4 \pi t)^{- \: \frac{n}{2}} e^{- \: f_\infty(t)}$. 
Then for all $t$, we have
$W_+(G_\infty(t),g_\infty(t),{f}_\infty(t),t) = W_\infty$.
In particular, $W_\infty < \infty$.

From (\ref{Wmon}), we obtain
\begin{align}
\triangle G_\infty - g_\infty^{yy} G_{\infty,y} G_\infty^{-1} 
G_{\infty,y} -
g_\infty^{yy} G_{\infty,y} {f}_{\infty,y} & = 0, \\
- \frac14 \Tr \left(
\left( G_\infty^{-1} G_{\infty,y} \right)^2 \right) + 
{f}_{\infty;yy} + \frac{1}{2t} g_{\infty,yy} & = 0. \notag
\end{align}
From \cite[Proposition 4.80]{L}, there is a traceless
symmetric matrix $X$ so that $g_{\infty,yy}(y,t) \: = \: \frac{t}{2} \: 
\Tr(X^2)$ and $G_{\infty}(y,t) = e^{yX}$. 
Here $y$ is a parametrization of $S^1$ whose
time-$1$ velocity is $L_\infty$.
For each $j$, we had $G_{s_j}(y+1, t) = H^T G_{s_j}(y,t) H$. Hence 
$G_{\infty}(y+1, t) = H^T G_{\infty}(y,t) H$ and so 
$e^X = H^T H$.
\end{proof}

We now prove part (ii) of Theorem \ref{thm1.6}.
From Proposition \ref{Wconv}, for any $K \in \Z^+$ and 
any $\epsilon > 0$,
there is some $t_0 < \infty$ such that
$\left( \frac{g_{yy}(y,t_0)}{t_0}, G(y,t_0) \right)$ is $\epsilon$-close
in the $C^K$-norm to 
$\left( \frac12 \Tr(X^2), e^{yX} \right)$. From the local 
stability result of
\cite[Theorem 3]{Knopf2}, after an overall reparametrization of
$S^1$, we have
\begin{align}
\lim_{t \rightarrow \infty} \frac{g_{yy}(y,t)}{t} & \: = \: \frac12 \: \Tr(X^2), \\
\lim_{t \rightarrow \infty} G(y,t) & \: = \: e^{yX}. \notag
\end{align}
The convergence is exponentially fast in the variable
$\ln (t)$, i.e. power-law fast in $t$.
(Strictly speaking, the result in \cite[Theorem 3]{Knopf2} is for
the modified Ricci flow (\ref{3.7}) but there is a similar
result for the unmodified Ricci flow (\ref{3.6}).)

This proves Theorem \ref{thm1.6}.

\begin{remark} \label{rem3.14}

Suppose that $H$ is parabolic. By Corollary \ref{cor3.2},
the fiber volumes are uniformly bounded above and below
by positive constants.
After pulling back to a double cover of the base $S^1$, if necessary,
we can assume that $\Tr(H) = 2$.
If $v \in \R^2$ is a nonzero $H$-invariant vector then
(\ref{3.22}) gives a uniform upper bound on the squared 
length $\langle v, G(y,t) v \rangle$.

Lemma \ref{lem3.7} implies that the length $L(t)$ of the base circle
is monotonically nondecreasing in $t$. We claim that
$\lim_{t \rightarrow \infty} \frac{L(t)}{\sqrt{t}} = 0$. If not then
we would conclude from the proof of Proposition \ref{Wconv} that
$N$ has a $\Sol$-structure, which contradicts the topological fact that it
has a $\Nil$-structure. 
  
If we knew that the diameters of the $T^2$-fibers
were $O(\sqrt{t})$ then we could conclude
from Lemma \ref{lem2.13} and \cite[Theorem 1.2.2]{L} 
that the pullback Ricci flow solution on the universal cover
$\widetilde{N}$ approaches the $\Nil$ expanding soliton. Based on
the calculation in the locally homogeneous $\Nil$ case, as 
in \cite[Subsubsection 3.3.3]{L0}, we expect that both $L(t)$ and
the fiber diameters are $O \left( t^\frac16 \right)$.
\end{remark}

\bibliography{warped100411}
\bibliographystyle{alpha}
\end{document}